\documentclass[12pt,epsf]{article}

\usepackage{amsfonts}
\usepackage{mathtools}
\usepackage{fancyhdr}
\usepackage[a4paper, top=2.5cm, bottom=2.5cm, left=2.2cm, right=2.2cm]{geometry}
\usepackage{times}
\usepackage[shortlabels]{enumitem}
\usepackage{amsmath}
\usepackage{changepage}
\usepackage{hyperref}
\usepackage{amssymb}
\usepackage{amsthm}
\usepackage{graphicx}
\usepackage{color}
\usepackage{authblk}
\usepackage{xcolor}
\usepackage{stfloats}
\usepackage{mathrsfs}
\usepackage{comment}
\usepackage{algorithm}

\theoremstyle{definition}
\newtheorem{theorem}{Theorem}
\newtheorem*{theorem-non}{Theorem}
\newtheorem{definition}{Definition}
\newtheorem{example}{Example}
\newtheorem{lemma}{Lemma}
\newtheorem{proposition}{Proposition}
\newtheorem{assumption}{Assumption}
\newtheorem{remark}{Remark}
\newtheorem{corollary}{Corollary}

\def\R{\mathbb R}
\def\G{\mathcal G}
\def\N{\mathbb N}
\def\E{\mathbb E}
\def\allone{\mathbf 1}
\def\pstar{\mathcal P^*}

\definecolor{blue}{RGB}{0,0,0}

\begin{document}

\title{A Perron-Frobenius Theorem for Strongly Aperiodic Stochastic Chains}

\author{Rohit Parasnis}
\author{Massimo Franceschetti}
\author{Behrouz Touri\thanks{Email: rohit100@mit.edu, massimo@ece.ucsd.edu, btouri@ucsd.edu\\ We thank Adel Aghajan for insightful discussions related to this work.}} 
\small\affil{Department of Electrical and Computer Engineering\\ University of California San Diego}\normalsize

\date{}
\maketitle
\sloppy

\begin{abstract}
    We derive a generalization of the Perron-Frobenius theorem to time-varying row-stochastic matrices as follows: using Kolmogorov’s concept of absolute probability sequences, which are time-varying analogs of principal eigenvectors, we identify a set of connectivity conditions that generalize the notion of irreducibility (strong
connectivity) to time-varying matrices (networks), and we
show that under these conditions, the absolute probability
sequence associated with a given matrix sequence is (a)
uniformly positive and (b) unique. Our results apply to both
discrete-time and continuous-time settings. We then discuss a few applications of our main results to non-Bayesian
learning, distributed optimization, opinion dynamics, and
averaging dynamics over random networks.
\end{abstract}

\section{Introduction}
The Perron-Frobenius theorem is a foundational tool in linear algebra that is central to the theory of Markov chains, and has many applications in database systems,
complex networks, population dynamics, opinion dynamics, social learning, economic growth and income inequalities, and many other physical, social, and economic phenomena~\cite{ballester2006s,proskurnikov2017tutorial,mui2002computational,ghiglino2010keeping, brin1998anatomy,golub2006arnoldi, newman2003structure,cull1973mathematical}.
Its strength lies in connecting the limiting behavior of $A^k$ as $k\to\infty$ with the structural (graph-theoretic) pattern of a fixed non-negative matrix $A$. For example, in the case of Google's PageRank algorithm, $A$ denotes the transition matrix of a Markov chain modeling a web-surfer, and the theory relates the ergodic (long term) behavior of this Markov chain to the centrality of webpages on World Wide Web (WWW).

{Unsurprisingly, there exists a large body of works that generalize the Perron-Frobenius theorem in a multitude of directions. Examples include~\cite{friedland2013perron},~\cite{aeyels2002extension},~\cite{sine1990nonlinear} and~\cite{pollicott1984complex}, which extend the classical theorem to polynomial maps with non-negative coefficients, nonlinear homogeneous systems, non-expansive maps, and complex
Perron-Frobenius type operators, respectively. A comprehensive treatment of  nonlinear extensions of the classical result can be found in~\cite{lemmens2012nonlinear}. Besides these extensions, the Perron-Frobenius theorem and its extension by Krein and Rutman~\cite{krein1948linear} to infinite-dimensional systems also find extensive application in the theory of monotone dynamical systems, which was pioneered in~\cite{hirsch1988stability} and treated extensively in the textbook~\cite{smith1995monotone}.} 

{The aforementioned tools and techniques have been applied to both static and time-varying dynamical systems. To add to this rich literature, therefore, we consider 
time-varying \textit{networked} dynamical systems 
and focus on the structures and patterns inherent in the {sequences} of network topologies that govern the dynamical behavior of such systems. Examples of such dynamics  include learning over time-varying social networks~\cite{golub2017learning}, distributed optimization and estimation over time-varying multi-agent networks~\cite{borkar1982asymptotic,nedic2014distributed},  distributed motion planning in robotic networks~\cite{bullo2009distributed}, etc. Although some of these are examples from distributed control where the relationships between distributed dynamics evolving over sequences of graphs and the connectivity conditions imposed on  the sequences  are  well-characterized, these connectivity conditions (e.g., $B$-connectivity/uniform strong connectivity~\cite{nedic2014distributed}) typically embody persistent or periodic connectivity and are not known to be \textit{necessary} (i.e., they are only known to be \textit{sufficient}) for the desired convergence properties of the concerned distributed algorithms.}

 {This paper is  a step towards filling these gaps in the literature. We focus on time-varying networks described by sequences of \textit{row-stochastic} matrices, which are central to numerous well-known applications of the Perron-Frobenius theorem (such as in the analysis of time-homogeneous Markov chains). 
 We extend two assertions of the classical theorem to a broad class of stochastic matrix sequences called strongly aperiodic stochastic chains. Our extensions result in (a) a time-varying analog of strong connectivity that is more general than standard connectivity notions for time-varying networks, 
 and (b) weak connectivity conditions that are sufficient to guarantee convergence to a steady state for distributed dynamics evolving over time-varying networks described by strongly aperiodic stochastic chains. Our contributions are as follows:\begin{enumerate} 
     \item We introduce \textit{approximate reciprocity}, a weak reciprocal connectivity condition that enables us to extend the concept of matrix {irreducibility} (which implies strong connectivity for static networks) to irreducibility for stochastic chains, which has the interpretation of strong connectivity \textit{over time} for time-varying networks. We show that our extension is more general than well-studied connectivity conditions such as B-connectivity and cut-balance~\cite{touri2013product} or instantaneous reciprocity.
     \item We find tight necessary and sufficient conditions for a strongly aperiodic stochastic chain to possess an \textit{absolute probability sequence } (Kolmogorov's time-varying analog of the Perron left eigenvector~\cite{kolmogoroff1936theorie}) that is unique and uniformly positive. These results (Theorems~\ref{thm:main} and~\ref{thm:uniqueness}) generalize two assertions of the classical theorem to time-varying networks described by strongly aperiodic chains.
     \item We then consider the continuous-time setting, where discrete sequences of row-stochastic matrices are replaced by continua of the Laplacian matrices of network digraphs. We provide an analog of approximate reciprocity for this setting and consequently derive the continuous-time analogs (Theorems~\ref{thm:ct_main} and~\ref{thm:ct_uniqueness}) of Theorems~\ref{thm:main} and~\ref{thm:uniqueness}. 
     \item We discuss a few applications of our main results that lead to novel insights into independent  random chains, opinion dynamics, and certain distributed algorithms.
 \end{enumerate}}

  The paper is organized as follows. We {introduce the technical background in Section~\ref{sec:prelim}},  formulate our problem and derive our main results in Section~\ref{sec:results}, explore some applications of our main results in Section~\ref{sec:applications}, and end with a few concluding remarks in Section~\ref{sec:conclusion}.
  
  \textit{Related Works:} Our work makes use of results from the theory of stochastic matrix sequences in relation to absolute probability sequences~{\cite{kolmogoroff1936theorie, chatterjee1977towards,touri2012product,touri2013product, bolouki2015eminence, bolouki2015consensus}} and the theory of non-homogeneous Markov chains~\cite{kolmogoroff1936theorie,chatterjee1977towards}. {The pioneering work~\cite{kolmogoroff1936theorie} introduced the concept of absolute probability sequences, showed that there exists an absolute probability sequence for every stochastic matrix sequence, and proved that the absolute probability sequence is unique if and only if the matrix sequence is \textit{ergodic}~\cite{chatterjee1977towards}, where ergodicity is a property studied in detail later in~\cite{chatterjee1977towards}. The results of~\cite{kolmogoroff1936theorie} and~\cite{chatterjee1977towards} are central to this paper, as we characterize the uniqueness of absolute probability sequences by connecting the concept of ergodicity with \textit{approximate reciprocity}, a concept we introduce in Section~\ref{sec:results}. 
  By introducing the infinite flow theory, the recent works~\cite{touri2012product,touri2013product}  extend the theory of stochastic matrix products by exploring the relationship between the asymptotic behavior of such products and the properties of network connections/influences evolving over time. These works also extend this framework to random stochastic chains and propose the following concepts that are related to the theoretical development of our main results: Class $\pstar$, which is the class of stochastic chains that admit a uniformly positive absolute probability sequence, infinite flow graphs, and  instantaneous reciprocity. In particular,~\cite{touri2013product} and \cite{touri2012product} introduce Class $\pstar$ and  show that many well-studied chains such as doubly stochastic chains and $B$-connected chains belong to this class. Then, they show that a condition, which they refer to as the \textit{infinite flow property}, is necessary and sufficient for the ergodicity of Class-$\pstar$-chains that satisfy a mild additional condition that resembles aperiodicity for time-homogeneous Markov chains. Finally, to close the loop, they show that instantaneous reciprocity is \textit{sufficient} for such chains to belong to Class $\pstar$. The current work extends these results as follows:}
  {\begin{enumerate} [leftmargin=0cm,itemindent=.5cm,labelwidth=\itemindent,labelsep=0cm,align=left]
      \item We significantly weaken the condition of instantaneous reciprocity, which requires the time-varying network to exhibit reciprocal connectivity/influence \textit{at every time instant}, to our condition of approximate reciprocity, which requires the network to exhibit a certain form of reciprocal connectivity \textit{over time}. Hence, our results apply to a much broader class of stochastic chains (see Remarks~\ref{rem:not_strong} and~\ref{rem:irreducible} and Examples~\ref{eg:one} and~\ref{eg:second} for more details).
      \item We show that  approximate reciprocity is not only sufficient, but also \textit{necessary} for a stochastic chain to belong to Class $\pstar$ (Theorems~\ref{thm:main} and~\ref{thm:ct_main}).
      
      Moreover, unlike~\cite{touri2012product} and~\cite{touri2013product}, we also derive a set of tight necessary and sufficient conditions for the uniqueness of absolute probability sequences (Theorems~\ref{thm:uniqueness} and~\ref{thm:ct_uniqueness}). These conditions are tight under approximate reciprocity (defined in Section~\ref{sec:results}) and a mild generalization of aperiodicity for stochastic chains.
      \item We also derive continuous-time analogs (Theorems~\ref{thm:ct_main} and~\ref{thm:ct_uniqueness}) of our main results, whereas the relevant results of~\cite{touri2012product} and~\cite{touri2013product} were developed only for the discrete-time setting.
  \end{enumerate}}
\noindent{Another related work is the work in~\cite{bolouki2015eminence}, which focuses primarily on the role of influential agent groups called \'{E}minence Grise Coalitions in driving continuous-time opinion dynamics to desired consensus states and provides a set of necessary and sufficient conditions for the existence of uniformly positive absolute probability sequences (i.e., chains belonging to Class $\pstar$). These conditions, though useful, do not lend themselves to simple interpretation. However, we use them as one of the many ingredients in our proofs of Theorems~\ref{thm:main} -~\ref{thm:ct_uniqueness}, which unravel the temporal network connectivity criteria that are equivalent to the abstract conditions in~\cite{bolouki2015eminence}. 

Related to both \cite{bolouki2015eminence} and our present work is~\cite{bolouki2015consensus}, which introduces a temporal connectivity condition called the \textit{infinite jet-flow property} and shows that this condition is equivalent to ergodicity if and only if the absolute probability sequence is uniformly positive. Our present work not only complements~\cite{bolouki2015consensus} by deriving tight necessary and sufficient conditions for the existence of a uniformly positive absolute probability sequence (via Theorem~\ref{thm:main}), but also provides a simpler characterization of ergodicity via Theorem~\ref{thm:uniqueness} (which focuses on uniqueness of the absolute probability sequence, which is in turn equivalent to ergodicity~\cite{kolmogoroff1936theorie}). Moreover, we tie uniqueness and uniform positivity of the absolute probability sequence together into a generalized notion of irreducibility by identifying a \textit{single} class of stochastic chains whose absolute probability sequences are both unique and uniformly positive. }

{Our work is also related to many existing results that extend Perron-Frobenius theory to nonlinear and/or time-varying systems. Among them,~\cite{deplano2020nonlinear} studies the convergence properties of positive systems by using non-linear Perron-Frobenius theory. Note that~\cite{deplano2020nonlinear}   focuses on static rather than time-varying non-linear state evolution maps. 
Similarly,~\cite{aeyels2002extension} extends the Perron-Frobenius theorem to a class of static and non-linear continuous-time systems that are positive and homogeneous. Another work that focuses on the continuous-time setting is~\cite{sanchez2009cones}, which uses certain extensions of the Krein-Rutman theorem~\cite{fusco1991perron,krasnoselskii1989positive} to study systems that can be considered monotone in a novel sense with respect to cones of rank {$k$ for a natural number $k$}. 

\textbf{{Terms and }Notation:} Let $\N$ denote the set of natural numbers, let $\N_0:=\N\cup\{0\}${, and for a given $n\in\N$, let $[n]:=\{1,2,\ldots,n\}$}. Let $\R$ denote the set of real numbers, $\R^n$ denote the set of $n$-dimensional real-valued column vectors, and let $\R^{n\times n}$ denote the set of $n\times n$ square matrices with real entries. For a  matrix $A\in\R^{n\times n}$, we let $a_{ij}=(A)_{ij}$ denote the entry in the $i$-th row and the $j$-th column of $A$. 

Let {$I_n$ (respectively, $O_n$)} denote the identity matrix {(respectively, the all-zeros matrix) in $\R^{n\times n}$, let $O_{m\times n}$ denote the all-zeros matrix in $\R^{m\times n}$}, let {$\mathbf 0_n\in\R^n$}  {(respectively, $\mathbf 1_n\in\R^n$)} denote the {$n$-dimensional} vector with all entries equal to zero {(one, respectively)}{, and let $e_n\in\R^n$ denote the $n$-th canonical basis vector, i.e., the vector with $1$ in its $n$-th entry and zeros in all other entries.}

{We assume that all matrix and vector inequalities hold entry-wise, e.g., $A\geq B$ means each entry of $A$ is no less than the corresponding entry of a matrix $B$ (of compatible dimension).} A vector $v\in\R^n$ is said to be \textit{stochastic} if $v$ is non-negative and $v^T\allone_n=1$. {A non-negative matrix $A\in \R^{n\times n}$ is said to be \textit{row-stochastic} (or simply \textit{stochastic}) if $A\allone_n=\allone_n$. In addition, $A$ is said to be \textit{substochastic} if $A\allone_n \le \allone_n$ entry-wise}.

Throughout the paper, we use $k$ as a discrete-time index that takes values in $\N_0$ (as in $\{A(k)\}_{k=0}^\infty$), and we use $t$ as a continuous-time index that takes values in $[0,\infty)$ (as in $\{A(t)\}_{t\ge 0}$). Let $\{A(k)\}_{k=0}^\infty$ be a {discrete-time} stochastic chain (a {discrete} sequence of row-stochastic matrices in $\R^{n\times n}$). Then, for any two times $k_1,k_2\in\N_0$ with ${k_1<k_2}$, we use $
    {A(k_2:k_1):=A(k_2-1)A(k_2-2)\cdots A(k_1)}
$ to denote the backwards matrix product of $\{A(k)\}_{k=0}^\infty$ over the time interval $[k_1,k_2]$ with the convention $A(k :k):=I_n$ for all $k\in\N_0$.  {In addition, we say that  $\{A(k)\}_{k=0}^\infty$ is a \textit{static chain} if $A(k)=A_0$ for all $k\in\N_0$ for a constant row-stochastic matrix $A_0\in\R^{n\times n}$}.

{For a vector $v\in\R^n$  and a subset $S\subset [n]$, we let $v_S\in \R^{|S|} $ denote the restriction of $v$ to the index set $S$. Similarly, } for a matrix $A\in\R^{n\times n}$, let $A_S$ be the principal sub-matrix of $A$ corresponding to the rows and columns indexed by $S$. Let $\bar S:=[n]\setminus S$, and let $A_{S\bar S}$ denote the sub-matrix of $A$ corresponding to the rows indexed by $S$ and the columns indexed by $\bar S$. For a sequence of matrices $\{A(k)\}_{k=0}^\infty$ in $\R^{n\times n}$ and times $k_0,k_1\in\N_0$ satisfying $k_0\leq k_1$,  let $A_S(k_1:k_0):=(A(k_1:k_0))_S$ and $A_{S\bar S}(k_1:k_0):=(A(k_1:k_0))_{S\bar S}$. 

An unweighted undirected graph with vertex set $[n]$ and edge set $E$ is denoted by $G=([n],E)$. On the other hand, a weighted time-varying directed graph with vertex set $[n]$, edge set $E(k)\subset[n]\times [n]$, and edge weights $\{w_{ij}(k): (i,j)\in [n]\times [n]\}$ is denoted by $G(k)=([n],E(k),W(k))$, where $W(k)\in\R^{n\times n}$ with ${(W(k))_{ij}:=w_{ij}(k)}$, which denotes the weight of the edge $(i,j)\in [n]\times [n]$. We assume that $w_{ij}(k)\neq 0$ if and only if $(i,j)\in E(k)$, i.e., $E(k)=\{(i,j)\in [n]\times [n]: w_{ij}(k)\neq 0 \}$. Recall that $G(k)$ is said to be \textit{strongly connected} if, for any two nodes $i,j\in[n]$, there exists a directed path from $i$ to $j$ in $G(k)$.

For a weighted time-varying directed graph ${G(t)=([n],E(t),W(t))}$ in continuous time, we let $L(t)=(\ell_{ij}(t))$ denote the weighted \textit{Laplacian} matrix of $G(t)$, defined by 
\begin{align*}
    \ell_{ij}(t)=
    \begin{cases}
        -w_{ij}(t)& \text{for all $i\neq j$},\\
        \sum_{q\neq i} w_{iq}(t)&\text{for $i=j\in[n]$}
    \end{cases}.
\end{align*} In addition, for a given non-negative matrix $A$, we let ${\G(A) = ([n],\mathcal E(A), A)}$ %
denote the weighted directed graph whose weighted adjacency matrix is $A$, i.e., we let ${\mathcal E(A)=\{(i,j)\in [n]\times [n] : A_{ij}>0\}}$.

\section{Preliminaries}~\label{sec:prelim}
In this section, we review the eigenvector assertions of the classical Perron-Frobenius theorem. {Recall that a non-negative matrix ${A_0\in\R^{n\times n}}$ is irreducible if its associated digraph $\G(A_0)$ is strongly connected.

Next, let us define the concept of \textit{instantaneous reciprocity} or \textit{cut-balance} and the \textit{infinite flow graph} of a stochastic chain, which we reproduce from~\cite{touri2012product,touri2013product} below.

\begin{definition}  [\textbf{Instantaneous Reciprocity/Cut-balance}] \label{def:reciprocity} A stochastic chain $\{A(k)\}_{k=0}^\infty$ is said to be \textit{cut-balanced} or \textit{instantaneously reciprocal} if there exists a constant $\alpha\in (0,1)$ such that 
\begin{align}\label{eqn:instareciproc}
     \sum_{i\in S} \sum_{j\in \bar S} a_{ij}(k) \geq \alpha \sum_{i\in \bar S}\sum_{j\in S} a_{ij}(k)  
\end{align}
holds for all times $k\in\N_0$ and all subsets $S\subset[n]$ and their complements $\bar S:=[n]\setminus S$. In other words, $\allone_n^T A_{S\bar S}(k)\allone_n\geq \alpha \allone_n^T A_{\bar S S}(k)\allone_n$ for all $S\subset[n]$ and all $k\in\N_0$.
\end{definition}

Intuitively, a stochastic chain is said to be instantaneously reciprocal if the associated sequence of directed graphs is such that the net influence of any subset $S$ of individuals on the complementary subset $\bar S$ is comparable to the net reverse influence of $\bar S$ on $S$, i.e., the ratio of the forward and the backward influences does not vanish in time.

\begin{definition} [\textbf{Infinite Flow Graph}~\cite{touri2012product}]
For a stochastic chain $\{A(k)\}_{k=0}^\infty$, its \emph{infinite flow graph} is the graph $G^\infty= ([n], E^\infty)$ with 
$$
    {E}^{\infty}:=\left\{\{i, j\} \,\Big \lvert\, \sum_{k=0}^{\infty}\left(a_{i j}(k)+a_{j i}(k)\right)=\infty, i \neq j \in[m]\right\}.
$$
\end{definition}

Intuitively, there exists a link from a node $j\in[n]$ to another node $i\in [n]\setminus\{j\}$ in the infinite flow graph $G^\infty$ if and only if  {either of the two} node{s $i$ and $j$} exerts a long-term influence on {the other node} in the time-varying directed graph $G(k)$ (whose weighted adjacency matrix at time $k$ is $A(k)$).

\begin{remark} [\textbf{Eigenvector Assertions of the Perron-Frobenius Theorem for Stochastic Matrices}]  \label{rem:original} 
{It was shown in \cite[Lemma 5.7]{touri2012product}} that {the concepts of infinite flow graph and instantaneous reciprocity are related to matrix irreducibility as follows}: a stochastic matrix $A_0$ is irreducible if and only if the corresponding static chain $\{A(k)=A_0\}_{k=0}^\infty$ is instantaneously reciprocal and its infinite flow graph $G^\infty$ is connected. { Using this characterization of irreducibility for stochastic matrices,}
the two eigenvector assertions of the Perron-Frobenius theorem can be restated  as follows: {For a static chain $\{A(k)\}_{k=0}^\infty$ with $A(k)=A_0\in\R^{n\times n}$ for all $k\in\N_0$, if the chain is instantaneously reciprocal and if its infinite flow graph is connected, then $A_0$ has a stochastic principal left eigenvector $\pi_0\in\R^n$ that is (a) entry-wise positive, and (b) unique.} 
\end{remark}

Note that the original theorem applies to left eigenvectors as well as to right eigenvectors. However,  the positivity assertion is trivial for the \textit{right} eigenvectors of stochastic matrices,  as all such matrices admit the all-one vector as a right eigenvector. Nonetheless, for such matrices, the implication of the assertion for their principal left eigenvectors is non-trivial and interesting. We now present an object that extends the notion of principal left eigenvectors to the case of row-stochastic chains.

\begin{definition}  [\textbf{Absolute Probability Sequence}~\cite{kolmogoroff1936theorie}] \label{def:abs_prob_seq} Let $\{A(k)\}_{k=0}^\infty$ be  a stochastic chain (sequence of row-stochastic matrices).
A sequence of stochastic vectors $\{\pi(k)\}_{k=0}^\infty$ is said to be an absolute probability sequence for $\{A(k)\}_{k=0}^\infty$ if 
$
\pi^T(k+1)A(k)=\pi^T(k)\quad\text{for all } k\in\N_0.
$
\end{definition}
Note that every stochastic chain admits an absolute probability sequence \cite{kolmogoroff1936theorie}. 
Moreover, if $\{A(k)\}_{k=0}^\infty$ is a static chain with $A(k)=A_0\in\R^{n\times n}$ for all $k\in\N_0$, then the static sequence $\pi(k)=\pi_0$, where $\pi_0\in\R^n$ is a stochastic vector satisfying $\pi_0^TA_0=\pi_0^T$,  is an absolute probability sequence for $\{A(k)\}_{k=0}^\infty$. Hence, absolute probability sequences are a  time-varying extension of stochastic principal left eigenvectors.

This discussion naturally leads to the following question: can we generalize  the eigenvector assertions of the Perron-Frobenius theorem (see Remark~\ref{rem:original}) to any class of non-static stochastic chains  using the notion of absolute probability sequences? {We answer this question in the next section using the following concept, which extends the notion of
positive principal left eigenvectors to the time-varying case.} }

\begin{definition} [\textbf{Class $\mathbf{\mathcal{P}^*}$}~\cite{touri2012product}] \label{def:pstar} We let (Class-)$\mathcal P^*$ be the set of all stochastic chains that admit uniformly positive absolute probability sequences, i.e., a sequence of stochastic vectors $\{\pi(k)\}_{k=0}^\infty$ such that $\pi(k)\geq p^*\mathbf 1_n$ for some scalar $p^*>0$ and all $k\in\N_0$. (Note that the absolute probability sequence and the value of $p^*$ may vary from chain to chain). 
\end{definition}

{It was  shown in~\cite{touri2013product} that Class $\pstar$ subsumes a well-studied class of stochastic chains called $B$\textit{-connected} chains, which was originally studied in~\cite{tsitsiklis1984problems}. We define this concept below.}

{\begin{definition}[\textbf{$B$-Connectivity}~\cite{touri2012product,touri2013product}]\label{def:unif_strong_connect}  A  stochastic chain $\{A(k)\}_{k=0}^\infty$ is said to be $B$-connected if
\begin{enumerate}
    \item there exists a $\delta>0$ such that for all $i,j\in[n]$ and all $k\in\N_0$, either $a_{ij}(k)\geq \delta$ or $a_{ij}(k)=0$,
    \item \label{item:strong_feedback} $a_{ii}(k)>0$ for all $i\in[n]$ and all $k\in \N_0$, and
    \item \label{item:b_strong} there exists a constant $B\in\N$ such that for the sequence of directed graphs  $\{G(k)=([n], E(k))\}_{k=0}^\infty$, where $ E(k) (i,j)\in [n]^2: a_{ji}(k)>0\}$, the graph
    $
        \G(k):=\left([n], \bigcup_{q=kB}^{(k+1)B-1} E(q)\right )
    $
    is strongly connected for every $k\in \N_0$.
\end{enumerate}
\end{definition}}
Intuitively, a stochastic chain is $B$-connected if the associated sequence of digraphs exhibits periodic connectivity.}

To extend the second of the two assertions of the classical theorem that we stated in Remark~\ref{rem:original} (the unique eigenvector assertion of  Perron-Frobenius theorem), we will need the following definitions.

\begin{definition} [\textbf{Ergodicity for Stochastic Chains}~\cite{chatterjee1977towards}]\label{def:ergodicity} A stochastic chain $\{A(k)\}_{k=0}^\infty\in\R^{n\times n}$ is said to be \textit{ergodic} if, for every $k_0\in\N$, there exists a stochastic vector $\pi(k_0)\in \R^n$ such that 
$\lim_{k\to\infty}A(k:k_0)=\allone_n\pi^T(k_0)$.
\end{definition}

To interpret the above definition, we first observe that in the distributed averaging dynamics $x(k+1)=A(k)x(k)$ with a starting time $k_0\in\N_0$ and an initial condition $x(k_0)\in\R^n$, we have 
$
    x(k)=A(k:k_0)x(k_0) \quad\text{for all}\quad k\in\N_0.
$
For an ergodic chain, this means that $\lim_{k\rightarrow\infty}x(k)=\pi^T(k_0) x(k_0)\allone_n$, which is a \textit{consensus} vector (i.e., all its entries are equal). Therefore, a stochastic chain being ergodic means that it always enables consensus regardless of the starting time $k_0$ and the starting point $x(k_0)$.

\begin{definition} [\textbf{Infinite Flow Stability}\label{def:inf_flow_stab}~\cite{touri2012product}]
A stochastic chain $\{A(k)\}_{k=0}^\infty$ is said to be \textit{infinite flow stable} if
\begin{enumerate}
    \item The sequence $\{x(k)\}_{k=k_0}^\infty$, which evolves as ${x(k+1)=A(k)x(k)}$, converges to a limit for all starting times $k_0\in\N_0$ and all initial conditions $x(k_0)\in\R^n$.
    \item $\lim_{k\to\infty}(x_i(k) - x_j(k))=0$ for all $(i,j)\in E^\infty$, where $E^\infty$ is the edge set of the infinite flow graph of $\{A(k)\}_{k=0}^\infty$.
\end{enumerate}

\end{definition}

Put simply, a stochastic chain is infinite flow stable if (a) the states of all the nodes of the corresponding time-varying network converge to a limit asymptotically in time, and (b) if a consensus is necessarily reached among nodes that exert a long-term influence on each other.

{Finally, we define \textit{strong aperiodicity}, a mild generalization of aperiodicity for stochastic chains. We assume strong aperiodicity in all our main results.
\begin{definition} [\textbf{Strong Aperiodicity}~\cite{touri2013product}]
    A stochastic chain $\{A(k)\}_{k=0}^\infty$ is  strongly aperiodic if there exists a $\gamma>0$ such that $A(k)\ge \gamma I$ for all $k\in \N_0$.
\end{definition}}

\section{Main Results}\label{sec:results}

We first extend the assertions of the Perron-Frobenius theorem that we stated in Remark~\ref{rem:original} to discrete-time stochastic chains of the form $\{A(k):k\in\N_0\}$ and then to continuous-time stochastic chains of the form $\{A(t):t\geq 0\}$.

\subsection{Discrete Time}

Since the definition of Class $\pstar$ eludes simple interpretation, we would like to derive necessary and sufficient conditions for a given stochastic chain to belong to Class $\pstar$. To this end, we introduce the idea of \textit{approximate reciprocity}, which is a weaker notion of reciprocity (Definition~\ref{def:reciprocity}).

\begin{definition} [\textbf{Approximate Reciprocity}] \label{def:approx_reciprocity}
A stochastic chain $\{A(k)\}_{k=0}^\infty$ is said to be \emph{approximately reciprocal} if there exist constants $p_0,\beta\in (0,\infty)$ such that for all $S\subset [n]$ and all times $0\leq k_0<k_1$, the following inequality holds 
\begin{align}\label{eq:crucial}
    p_0\sum_{k=k_0}^{k_1-1}  \allone_{|S|} ^T A_{S\bar S}(k)\allone_{|\bar S|} \leq \sum_{k=k_0}^{k_1-1} \allone_{|\bar S|} ^T A_{\bar S S}(k)\allone_{|S|} +\beta.
\end{align}
\end{definition}

{Intuitively, a stochastic chain is said to be approximately reciprocal if the associated sequence of  digraphs is such that the net influence of any subset $S$ of individuals on the complementary subset $\bar S$ is, \textit{up to a slack parameter $\beta$}, comparable to the net reverse influence of $\bar S$ on $S$ \textit{over time}. Note that instantaneous reciprocity (Definition~\ref{def:reciprocity}) is a special case of approximate reciprocity in which $\beta=0$.}

{\begin{remark}\label{rem:not_strong}
Approximate reciprocity may appear to be a restrictive condition because it requires~\eqref{eq:crucial} to hold for \emph{all} times $k_0,k_1\in\N$ with $k_0<k_1$. On the contrary, as we argue below, this concept is general enough to apply to a large class of stochastic chains. 

Note that the slack parameter $\beta$ is an arbitrary positive constant. Therefore, whenever there exists a time-invariant upper bound on the difference between the total forward influence of $S$ on $\bar S$ and a non-vanishing fraction of the total reverse influence of $\bar S$ on $S$ over a finite time interval, approximate reciprocity holds irrespective of the value of the upper bound. This is clarified further by the examples below.
\end{remark}
\begin{example}\label{eg:one}
     Let $U,L\in\R^{n\times n}$ be defined as $U:=(1/2)\left(I_n + \allone e_n^T \right)$ and $L:=I_n - e_n (e_n - (1/n) \mathbf 1)^T$, so that $U$ (respectively, $L$) is upper-triangular (respectively, lower-triangular) and row-stochastic with positive diagonal entries. Then it can be verified that the stochastic chain  $\{A(k)\}_{k=0}^\infty$ defined by $A(2^\ell)= U$ and $A(2^\ell+1) = L$ for all $\ell\in\N_0$, and $A(k)=I_n$ for all $k\in\N_0\setminus\{2^0, 2^0+1, 2^1, 2^1+1,\ldots\}$ is approximately reciprocal {with $p_0=2/n$ and $\beta = {n}/{2}$} (see Definition~\ref{def:approx_reciprocity}).  However, the chain is not B-connected because the off-diagonal entries are all zero over  time intervals of exponentially increasing lengths, and it is neither instantaneously reciprocal as $U,L$ do not satisfy~\eqref{eqn:instareciproc}.
 \end{example}}

{The above example shows that not all approximately reciprocal chains are $B$-connected or instantaneously reciprocal. On the other hand, it can be verified that any $B$-connected chain is approximately reciprocal.}

{We now give another example to compare approximate reciprocity with instantaneous reciprocity.}

{\begin{example}\label{eg:second}
    Consider the dynamics of belief aggregation over a network of $n$ sensors. Suppose that every two sensors that can communicate with each other do so via a semi-duplex communication channel that enables asynchronous rather than simultaneous bidirectional communication. Suppose the sensors aggregate their neighbors' beliefs using weighted averaging, and let $a_{ij}(k)$ denote the weight assigned by sensor $j\in[n]$ to sensor ${i\in[n]}$ in the $k$-th aggregation round. In addition, suppose the aggregation weights $\{a_{ij}(k):i,j\in[n]\}$ are all bounded away from $0$ (i.e., there exists a constant $\delta>0$ such that $a_{ij}(k)\ge \delta$ whenever $a_{ij}(k)\ne 0$). As simultaneous bidirectional communication is not possible, we have $a_{ij}(k)=0$ whenever $a_{ji}(k)\ge \delta$.

    Suppose there exists a  $T\in\{2,3,\ldots\}$ such that for every $T$ transmissions from any sensor $i\in[n]$ to another sensor ${j\in[n]\setminus\{i\}}$ during any given time interval, there occurs at least one transmission from $j$ to $i$ during the same interval. In other words, the frequency of communication in any one direction is at least $(1/T)$-th of the frequency of communication in the reverse direction. Then, regardless of the value of $T$, it can be shown that $\{A(k)\}_{k=0}^\infty$ is approximately reciprocal with $p_0=\delta/T$ and $\beta = (1-1/T)\delta n^2${, i.e., for any set $S\subset [n]$ and any two times $k_0,k_1\in \N_0$ with $k_1> k_0$, we have 
    \begin{align}\label{eq:extra}
       \delta T^{-1} \sum_{k=k_0}^{k_1-1}  \allone_{|S|} ^T A_{S\bar S}(k)\allone_{|\bar S|} &\leq \sum_{k=k_0}^{k_1-1} \allone_{|\bar S|} ^T A_{\bar S S}(k)\allone_{|S|}\cr
       &\quad + (1-T^{-1})\delta n^2.
    \end{align}
    For the proof of~\eqref{eq:extra}, see Appendix~\ref{app:nine}.}
\end{example}}

 {The above example shows that  approximate reciprocity applies to scenarios in which a subset of agents $S\subset [n]$ exert a one-way influence on the complementary subset $\bar S:=[n]\setminus S$ over arbitrarily long intervals of time (i.e., for arbitrarily large values of $T$ in the context of Example~\ref{eg:second}), as long as the lengths of these intervals are bounded in time (so that $T<\infty$). 
 }

{Thus, the class of approximately reciprocal chains is significantly broader than that of instantaneously reciprocal chains.} 

We now show that approximate reciprocity is a necessary condition for a given stochastic chain to belong to Class $\pstar$.

\begin{proposition} [\textbf{Necessary Conditions for Class $\pstar$}] \label{prop:nec_cond}
Let $\{A(k)\}_{k=0}^\infty$ be a stochastic chain in $\R^{n\times n}$ that belongs to Class $\pstar$. Then, $\{A(k)\}_{k=0}^\infty$ is approximately reciprocal.
\end{proposition}

\begin{proof}
Consider any set $S\subset [n]$, and let $\bar S:=[n]\setminus S$. Then, there exists a permutation matrix $Q$ such that
$$
    Q^T A(k) Q = 
    \begin{bmatrix}
        A_S(k)&A_{S\bar{S}}(k)\\A_{\bar{S}S}(k)&A_{\bar{S}}(k)
    \end{bmatrix}
$$   
for all $k\in\N_0$.
Let $\{\pi(k)\}_{k=0}^\infty$ denote an absolute probability sequence for $\{A(k)\}_{k=0}^\infty$. Then one may verify that the corresponding absolute probability sequence for $\{Q^T A(k) Q\}_{k=0}^\infty$ is given by $\{ \tilde\pi(k)\}_{k=0}^\infty$, where $\tilde \pi(k):=[\pi_S(k)\,\,\pi_{\bar S}(k)]^T$
for all $k\in\N_0$. As a result, the following holds for all $k\in\N_0$:
\begin{gather*}
    \begin{bmatrix}
        \pi_S^T(k+1)\quad \pi_{\bar S}^T(k+1) 
    \end{bmatrix}
    \begin{bmatrix}
        A_S(k)&A_{S\bar{S}}(k)\\A_{\bar{S}S}(k)&A_{\bar{S}}(k)
    \end{bmatrix} \\= 
    \begin{bmatrix}
        \pi_S^T(k)\quad \pi^T_{\bar S}(k).
    \end{bmatrix}
\end{gather*}
The above equation is a pair of two vector equations, one of which is
$
    \pi_S^T(k+1) A_{S\bar S}(k) + \pi^T_{\bar S}(k+1) A_{\bar S}(k) = \pi^T_{\bar S}(k).
$
Multiplying each side of this equation by the all-ones vector yields
\begin{align}\label{eq:beh_first}
    \pi_S^T(k+1) A_{S\bar S}(k)\allone_{|\bar S|} + \pi^T_{\bar S}(k+1) A_{\bar S}(k)\allone_{|\bar S|}  = \pi^T_{\bar S}(k)\allone_{|\bar S|} .
\end{align}
On the other hand, the row-stochasticity of $A(k)$ implies that
\begin{align}\label{eq:beh_second}
    A_{\bar S}(k)\allone_{|\bar S|}  = \allone_{|\bar S|}  - A_{\bar S S}(k)\allone_{|\bar S|} .
\end{align}
Combining~\eqref{eq:beh_first} and~\eqref{eq:beh_second} gives us
\begin{align}\label{eq:beh_third}
    &\pi_S^T(k+1) A_{S\bar S}(k)\allone_{|\bar S|}  + \pi^T_{\bar S}(k+1) (\allone_{|\bar S|}  - A_{ \bar S S }(k)\allone_{| S|}  )\nonumber\\
    &= \pi^T_{\bar S}(k)\allone_{|\bar S|}.
\end{align}
{On subtracting $\pi^T_{\bar S}(k+1) (\allone_{|\bar S|}  - A_{ \bar S S }(k)\allone_{| S|}  )$ from both sides, we obtain}
\begin{align*}
    \pi_{S}^T(k+1)A_{S\bar{S}}(k)\allone_{|\bar S|}  &=\left(\pi_{\bar{S}}^T(k)-\pi_{\bar{S}}^T(k+1)\right)\allone_{|\bar S|} \cr
    &\quad+\pi_{\bar{S}}^T(k+1)A_{\bar{S}S}(k)\allone_{|\bar S|} 
\end{align*}
Since $\{A(k)\}_{k=0}^\infty\in\pstar$, there exists a $p^*>0$ such that ${\pi_{S}(k+1)\geq p^*\allone_{| S|}  }$. Therefore, 
\begin{align}\label{eq:beh_last}
    &p^*\allone_{| S|}  ^T A_{S\bar{S}}(k)\allone_{|\bar S|} \cr
    &\leq \left(\pi_{\bar{S}}^T(k)-\pi_{\bar{S}}^T(k+1)\right)\allone_{|\bar S|} +\pi_{\bar{S}}^T(k+1)A_{\bar{S}S}(k)\allone_{| S|} \cr
    &\leq \left(\pi_{\bar{S}}^T(k)-\pi_{\bar{S}}^T(k+1)\right)\allone_{|\bar S|} +\allone_n^TA_{\bar{S}S}(k)\allone_{|S|} .
\end{align}
Now, let $k_0,k_1\in\N$ be any two numbers such that $k_0<k_1$. Then, summing both the sides of~\eqref{eq:beh_last} over the range ${k\in \{k_0,k_0+1,\ldots, k_1-1\}}$ yields
\begin{align*}
    p^* \sum_{k=k_0}^{k_1-1} \allone_{| S|} ^T A_{S\bar{S}}(k)\allone_{|\bar S|} 
    &\leq \left(\pi_{\bar{S}}^T(k_0)-\pi_{\bar{S}}^T(k_1)\right)\allone_{|\bar S|} \cr
    &\quad+ \sum_{k=k_0}^{k_1-1} \allone_{| S|} ^TA_{\bar{S}S}(k)\allone_{|\bar S|} ,
\end{align*}
where we have used a telescoping sum on the right hand side. Since 
$$
     \left(\pi_{\bar{S}}^T(k_0)-\pi_{\bar{S}}^T(k_1)\right)\allone_{|\bar S|} \leq  \pi_{\bar{S}}^T(k_0)\allone_{|\bar S|}\leq \pi^T(k_0) \allone_{|\bar S|}  = 1,
$$
the above implies that $p^* \sum_{k=k_0}^{k_1-1} \allone_{| S|}^T A_{S\bar{S}}(k)\allone_{|\bar S|}$ is no greater than $1 +  \sum_{k=k_0}^{k_1-1} \allone_{|\bar S|}^T A_{\bar{S}S}(k)\allone_{| S|}$. We have thus proved~\eqref{eq:crucial} for $\beta=1$. This completes the proof.
\end{proof}

Interestingly, as we will show, approximate reciprocity is also a sufficient condition for  strongly aperiodic  chains to belong to Class $\pstar$. 

To connect approximate reciprocity, a property expressed in terms of \textit{sums} of matrix entries, to Class $\pstar$, a concept defined using \textit{products} of matrices, we need the following lemma that help relate matrix sums to matrix products.

\begin{lemma} \label{lem:uncertain}
Let $n,\sigma\in \N$ and $i,j\in [n]$ be given. Let $\{B(k)\}_{k=0}^{\sigma-1}$ be a sequence of substochastic matrices in $\R^{n\times n}$, and let $k_L:=\max\{k\in\{0,1,\ldots,\sigma-1\}:B_{ji}(k)>0\}$. Suppose there exist positive constants $\eta_i$ and $\eta_j$ such that
\begin{align*}
    B_{ii}(k_1:k_0)&\geq \eta_i \quad \text{if}\quad 0\leq k_0\leq k_1\leq  k_L,\cr
    B_{jj}(k_1:k_0)&\geq \eta_j \quad \text{if}\quad 0\leq k_0\leq k_1\leq  \sigma,\text{ and}\cr
    \sum_{k=0}^{\sigma-1}B_{ji}(k)&\geq \delta\quad \text{for some}\quad \delta\in (0, \eta_j).
\end{align*}
Then $B_{ji}(\sigma:0)\geq \frac{1}{2}\eta_i\eta_j \delta$.
\end{lemma}

{The proof of Lemma~\ref{lem:uncertain} is relegated to Appendix~\ref{app:A}.} In addition to the above lemmas, we need the notion of approximately stochastic chains, which we define below.

\begin{definition} [\textbf{Approximate Stochasticity}] Let $n\in \N$ and $m\in\N\cup\{\infty\}$ be given. A sequence $\{A(k)\}_{k=0}^m$ of $n\times n$ substochastic matrices is said to be \emph{approximately stochastic} if there exists a constant $\Delta<\infty$ such that 
\begin{align}\label{eq:approx_stoch}
    \sum_{k=0}^m \allone_n^T(\allone_n-A(k)\allone_n)\leq \Delta.
\end{align}
The constant $\Delta$ will be referred to as the \textit{deviation from stochasticity} of the sequence $\{A(k)\}_{k=0}^m$.
\end{definition}

We are now well-equipped to establish approximate reciprocity as a sufficient condition for strongly aperiodic chains to lie in $\pstar$. To do so, we use inductive arguments involving approximately stochastic chains to prove a slightly more general result that asserts that the backward matrix products of the concerned chains can be uniformly lower-bounded by a multiple of the identity matrix. {We prove this general result below after introducing the required notation.}

\def\A{\mathcal A}

{For each $n\in\N$, let $\A_n(\gamma,p_0,\beta,\Delta)$ denote the family of substochastic chains
$\{A(k)\}_{k=0}^\infty$ in $\R^{n\times n}$ that satisfy} 
\begin{enumerate}\item  \textbf{(Strong aperiodicity/Feedback property}: ${a_{ii}(k)\geq \gamma}$ for all $i\in [n]$ and all $k\in\N_0$ for $\gamma\in(0,1)$, 
\item \textbf{(Approximate reciprocity)}: \eqref{eq:crucial}  holds for every subset $S\subset [n]$ and $k_0,k_1\in \N_0$ satisfying $k_0< k_1$ for $p_0\in (0,1)$, $\beta\in (0,\infty)$, and
\item \textbf{(Approximate stochasticity)}: $\{A(k)\}_{k=0}^\infty$ satisfies~\eqref{eq:approx_stoch} with $\Delta\in [0,\infty)$ being is its deviation from stochasticity.
\end{enumerate}

\def\A{\mathcal A}

\begin{proposition}\label{prop:main}
There exists a {continuous} function
$
    \eta_n: (0,1)\times (0,1)\times (0,\infty)\times [0,\infty)\rightarrow (0,1)
$
{such that for any combination of parameters $n\in \N$, $\gamma\in (0,1)$, $p_0\in (0,1)$, $\beta\in (0,\infty)$, and ${\Delta\in [0,\infty)}$, every {sub}stochastic chain ${\{A(k)\}_{k=0}^\infty\in \A_n(\gamma,p_0,\beta,\Delta)}$ satisfies
$$
    A(k_1:k_0)\geq \eta_n(\gamma,p_0,\beta,\Delta)I_n
$$
for all $k_0,k_1\in\N_0$ with $k_0<k_1$.}
\end{proposition}

{The proof of Proposition~\ref{prop:main} is relegated to Appendix~\ref{app:B}.} 

We now obtain the desired sufficient conditions as a straightforward consequence of the above proposition.

\begin{theorem-non}  [\textbf{Sufficient Conditions for Class $\pstar$}] \label{cor:suff_cond}
Suppose $\{A(k)\}_{k=0}^\infty$ is a strongly aperiodic stochastic chain, i.e., suppose there exists a $\gamma>0$ such that $A(k)\geq \gamma I_n$ for all $k\in\N_0$. If $\{A(k)\}_{k=0}^\infty$ is approximately reciprocal, then $\{A(k)\}_{k=0}^\infty\in\pstar$. 
\end{theorem-non}

\begin{proof}
Since $\{A(k)\}_{k=0}^\infty$ is a stochastic chain, it satisfies approximate stochasticity (with the deviation from stochasticity being $\Delta=0$). 
Hence, if $\{A(k)\}_{k=0}^\infty$ satisfies~\eqref{eq:crucial} for all $S\subset [n]$ and all $k_0,k_1\in\N_0$ with $k_0<k_1$, then it follows from Proposition~\ref{prop:main} that there exists an $\eta>0$ satisfying $A(k_1:k_0)\geq \eta I$ for all $k_0,k_1\in\N_0$ with $k_0\leq k_1$. This means that $\allone_n^T A(k_1:k_0)\geq \eta \allone_n^T$ for all $k_1, k_0\in\N_0$. In light of Lemma 8 of \cite{bolouki2015eminence}, this means that $\{A(k)\}_{k=0}^\infty\in\pstar$.
\end{proof}

As a direct consequence of the above result and Proposition~\ref{prop:nec_cond}, we obtain the following necessary and sufficient conditions for Class $\pstar$: a strongly aperiodic stochastic chain belongs to $\pstar$ iff it is approximately reciprocal. Since a stochastic chain belongs to Class $\pstar$ iff it has a uniformly positive absolute probability sequence, we have the following result.

\begin{theorem} [\textbf{An Analog of the Positive-Eigenvector Assertion of the Perron-Frobenius Theorem}]\label{thm:main}
Suppose $\{A(k)\}_{k=0}^\infty$ is a strongly aperiodic stochastic chain. Then $\{A(k) \}_{k=0}^\infty$ has a uniformly positive absolute probability sequence if and only if it is approximately reciprocal.
\end{theorem}

Observe how Theorem~\ref{thm:main} parallels the first of the two assertions of the classical theorem that we stated in Remark~\ref{rem:original}. This assertion states that for a static network that is reciprocal and whose infinite flow graph is connected (i.e., a network defined by an irreducible matrix), there exists a positive principal left eigenvector. Analogously, Theorem~\ref{thm:main} asserts that for a dynamic network that is approximately reciprocal, there exists a uniformly positive absolute probability sequence.

{We now extend the second of the two assertions of the classical theorem that we stated in Remark~\ref{rem:original}.}

We are now ready to {state and prove the next main result}.

\begin{theorem} [\textbf{An Analog of the Uniqueness Assertion of the Perron-Frobenius Theorem}] \label{thm:uniqueness} Let $\{A(k)\}_{k=0}^\infty$ be a strongly aperiodic stochastic chain that is also approximately reciprocal. Then, $\{A(k)\}_{k=0}^\infty$ admits a unique absolute probability sequence if and only if its infinite flow graph is connected.
\end{theorem} 

\begin{proof}
From Theorem~\ref{thm:main}, we know that $\{A(k)\}_{k=0}^\infty$ admits a uniformly positive absolute probability sequence, i.e., $\{A(k)\}_{k=0}^\infty\in\pstar$.  As a result, Theorem 4.4 of~\cite{touri2012product} implies that $\{A(k)\}_{k=0}^\infty$ is infinite flow stable. 

Now, suppose that the infinite flow graph of $\{A(k)\}_{k=0}^\infty$ is connected. Then we know from Lemma~\ref{lem:obv} that $\{A(k)\}_{k=0}^\infty$ is ergodic. It  now follows from Theorem 1 of~\cite{blackwell1945finite} that $\{A(k)\}_{k=0}^\infty$ has a unique absolute probability sequence.

On the other hand, suppose that the infinite flow graph of  $\{A(k)\}_{k=0}^\infty$ is not connected. Then, by Lemma 3.6 of~\cite{touri2012product}, either there exists an initial condition $(k_0,x(k_0))$ with $k_0\in\N$ and $x(k_0)\in\R^n$ such that $x(k+1)=A(k)x(k)$ does not converge to a steady state (Case 1: $\lim_{k\rightarrow\infty} x(k) $ does not exist), or there exist indices $i$ and $j$ such that $(i,j)\in [n]\times [n]$ and $\limsup_{k\rightarrow\infty} |x_i(k)-x_j(k)|> 0$ (Case 2).

In the first case, we know that $\lim_{k\rightarrow\infty} A(k:k_0)$ does not exist (for otherwise, $\lim_{k\rightarrow\infty} x(k_0)=\lim_{k\rightarrow\infty }A(k:k_0)x(k_0)$ would exist). Hence, $\{A(k)\}_{k=0}^\infty$ is not ergodic.

Consider now the second case and suppose that $\{A(k)\}_{k=0}^\infty$ is ergodic. Then, for every initial condition $(k_0,x(k_0))$, there exists a $\pi(k_0)\in\R^n$ such that $\lim_{k\to\infty} x(k) = \lim_{k\to\infty} A(k:k_0) x(k_0)= \pi^T(k_0) x(k_0)\allone_n$, which implies that $\lim_{k\rightarrow \infty} x_l(k)=\lim_{k\to\infty}x_m(k)$ for all $l,m\in[n]$. However, this contradicts the hypothesis of Case 2. Hence, $\{A(k)\}_{k=0}^\infty$ cannot be ergodic.

We have thus shown that if the infinite flow graph of $\{A(k)\}_{k=0}^\infty$ is not connected, it is not ergodic. It now follows from Theorem 1  in~\cite{blackwell1945finite} that if the infinite flow graph of $\{A(k)\}_{k=0}^\infty$ is not connected, then the chain does not admit a unique absolute probability sequence.
\end{proof}

Theorem~\ref{thm:uniqueness}  parallels the uniqueness assertion of the Perron-Frobenius theorem. In view of Remark~\ref{rem:original}, the classical theorem asserts that, if a matrix describes a static network that is reciprocal and whose infinite flow graph is connected, then its principal left eigenvector is unique. Analogously, Theorem~\ref{thm:uniqueness} asserts that, if a stochastic chain describes a time-varying network that is approximately reciprocal and whose infinite flow graph is connected,  its absolute probability sequence is unique.
 
 Besides, it is worth noting that approximately reciprocal chains whose infinite flow graphs are connected are a time-varying analog of irreducible matrices{. This is because, as shown in~\cite{touri2012product,touri2013product}, every static irreducible chain (i.e., every stochastic chain $\{A(k)\}_{k=0}^\infty$ for which there exists  an irreducible matrix $A_0$ such that  $A(k)=A_0$ for all $k\in\N_0$) is approximately reciprocal with a connected infinite flow graph. Therefore, we shall henceforth use the term \textit{irreducible chains} to refer to (either static or non-static) stochastic chains that are approximately reciprocal with connected infinite flow graphs. 
 
{\begin{remark} \label{rem:irreducible} Recall the  stochastic chain $\{A(k)\}_{k=0}^\infty$ defined in Example~\ref{eg:one}. It can be verified that the infinite flow graph of this chain is connected. It follows from Example~\ref{eg:one} that $\{A(k)\}_{k=0}^\infty$ is irreducible but not $B$-connected. Hence, our notion of irreducibility is more general than that of $B$-connectivity.
\end{remark}}

}

\begin{figure*}[h]
    \centering
    \includegraphics[height = 2.8 cm]{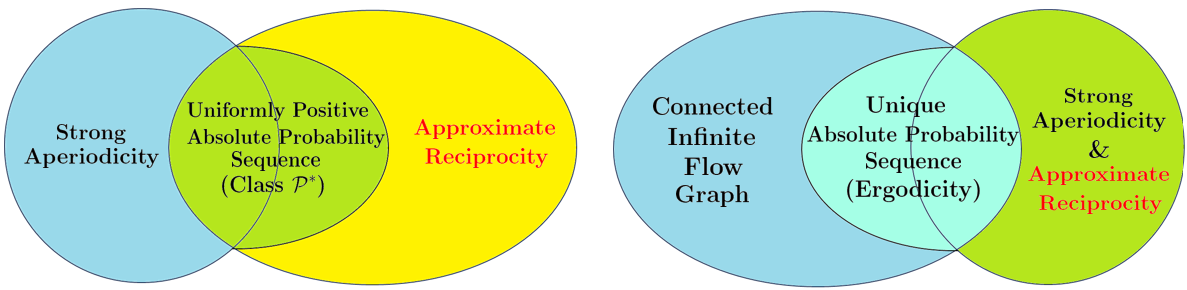}
    \caption{{\small Venn diagrams illustrating the relationships between the key concepts
    }
    }
\end{figure*}
\subsection{Some Interpretations of the Main Results}
To interpret Theorems~\ref{thm:main} -~\ref{thm:ct_uniqueness}, we start from some existing interpretations of the assertions of the classical theorem that we stated in Remark~\ref{rem:original}, and we extend these interpretations to the case of time-varying networks.
\begin{enumerate}
    \item \textit{Markov Chains:} The eigenvector assertions of the Perron-Frobenius theorem can be interpreted as follows: for a time-homogeneous Markov chain with transition probabilities given by an irreducible matrix, the probability of visiting any given state converges asymptotically in time to a unique positive value, regardless of the initial probability distribution. Analogously, Theorems~\ref{thm:main} and~\ref{thm:uniqueness} can be interpreted as follows: given a starting time, for a backward-propagating time-non-homogeneous Markov chain with transition probabilities given by an irreducible, strongly aperiodic chain, the probability of visiting any given state converges asymptotically in time to a unique positive value, regardless of the initial probability distribution. Although this limiting probability is a function of the starting time, it is bounded away from zero by a fixed threshold that does not depend on the starting time.
    \item \textit{Opinion Dynamics:} In the context of opinion dynamics, the matrix $A(k)$ can be interpreted as the \textit{influence matrix} at time $k$, i.e., $a_{ij}(k)$ quantifies the extent to which agent $i$ values agent $j$'s opinion at time $k$ (or equivalently, the extent to which agent $j$ influences agent $i$ at time $k$). Therefore, an irreducible chain (and hence also an irreducible matrix) describes a network in which every subset of agents influences the complementary subset persistently over the entire course of opinion evolution, which means that there exists no group of elite agents that dominate others forever. Additionally, as mentioned before, absolute probability sequences can be interpreted as quantifying the agents' social powers.
    
    Therefore, an interpretation of the eigenvector assertions of the original theorem is as follows: in a static social network, the social power of every agent (given by the eigenvector centrality of the corresponding network node) is unique and positive if no subset of agents dominates other agents forever. Analogously, Theorems~\ref{thm:main}-~\ref{thm:uniqueness} can be interpreted as follows: in a time-varying social network, the time-varying social power of every agent (given by the Kolmogorov centrality of the corresponding network node) is unique and uniformly positive (lower-bounded by a constant positive threshold)  if no subset of agents dominates other agents forever.
\end{enumerate}

\subsection{Continuous Time}

We now extend our discrete-time results (Theorems~\ref{thm:main} and~\ref{thm:uniqueness}) to continua of row-stochastic matrices, henceforth called continuous-time stochastic chains. Consider the following continuous-time analog of the discrete-time dynamics $x(k+1)=A(k)x(k)$:
\begin{align}\label{eq:continuous_time_dynamics}
    \dot x(t) = A(t)x(t)\quad \text{for all }t\geq 0,
\end{align}
where $A(t)=-L(t)$ is the negative of the Laplacian matrix of a given digraph $G(t)$. Throughout this section, we assume
\begin{align}\label{eq:assumption}
    \int_{t_1}^{t_2} a_{ij}(t)dt<\infty \quad \text{for all }0\leq t_1<t_2<\infty.
\end{align}
It is well-known~\cite{martin2015continuous,sontag2013mathematical} that under Assumption~\eqref{eq:assumption}, the solution to~\eqref{eq:continuous_time_dynamics} is unique and can be expressed as
\begin{align}
    x(t) = \Phi(t,\tau)x(\tau)\quad\text{for all } t\geq \tau \geq 0,
\end{align}
where the \textit{state-transition matrix} $\Phi$ is the unique solution to the equation continuum $\Phi(t,\tau) = I + \int_\tau^{t} A(\tau')\Phi(\tau',\tau) d\tau'$ for all
$t\geq \tau \geq 0$.

It is also known that
\begin{align}
    \Phi(t_2,t_1)=\Phi(t_2,\tau)\Phi(\tau,t_1)\quad\text{for all }t_2\geq \tau\geq t_1\geq 0
\end{align}
and that $\Phi(\tau,\tau)=I_n$ for all $\tau\geq 0$.
{Moreover, ${A(t)=-L(t)}$ implies that} $\Phi(t,\tau)$ is row-stochastic for all ${t\geq \tau \geq 0}$. Therefore, for any sequence of increasing times $\{t_k\}_{k=1}^\infty$ in $[0,\infty)$, if we let $B(k):=\Phi(t_{k+1},t_k)$ for all $k\in\N_0$, then we have $B(m:\ell)=\Phi(t_m:t_\ell)$ for all $\ell,m\in\N$ with $\ell\leq m$. As a result, an application of Proposition~\ref{prop:main} to the stochastic chain $\{B(k)\}_{k=0}^\infty$ yields the following result.  
\begin{lemma}\label{lem:state_transition_matrix_suff} Let $\Phi(\cdot,\cdot)$ denote the state transition matrix for the dynamics~\eqref{eq:continuous_time_dynamics} under the assumption~\eqref{eq:assumption}. Consider now a sequence of increasing times $\{t_k\}_{k=0}^\infty$ in $[0,\infty)$ and a constant $\gamma>0$ such that $\Phi(t_{k+1},t_k)\geq \gamma I$ for all $k\in\N_0$. If there exist constants $\tilde p_0,\tilde \beta\in(0,\infty)$ such that 
\begin{align}\label{eq:approx_reciprocity_state}
   &\tilde p_{0} \sum_{k=\ell}^{m}\allone_{|S|}^T \Phi_{S\bar S}\left(t_{k+1}, t_{k}\right)\allone_{|\bar S|} \cr 
   &\qquad \leq \sum_{k=\ell}^{m}\allone_{|\bar S|}^T \Phi_{\bar S S}\left(t_{k+1}, t_{k}\right)\allone_{|S|}+\tilde\beta
\end{align}
holds for all sets $S\subset [n]$ and for all $\ell,m\in\N_0$ with $\ell\leq m$, then there exists an $\eta>0$ such that $\Phi(t_m,t_\ell)\geq \eta I_n$ for all $\ell,m\in\N_0$ satisfying $\ell\leq m$.
\end{lemma}

It is clear from Lemma~\ref{lem:state_transition_matrix_suff}  and from~\cite[Lemma 8]{bolouki2015eminence} that the discrete-time chain $\{\Phi(t_{k+1},t_k)\}_{k=0}^\infty$ lies in Class $\pstar$ if approximate reciprocity~\eqref{eq:approx_reciprocity_state} and the strong aperiodicity condition $\Phi(t_{k+1},t_k)\geq \gamma I_n$ are satisfied. The following assumptions ensure that both these conditions are met.

\begin{assumption}[\textbf{Uniform Bound on Integral Weights}~\cite{martin2015continuous}] \label{assume:uniform_bound} There exists an $M<\infty$ and an increasing sequence $\{t_k\}_{k=0}^\infty$ in $[0,\infty)$ such that
$
    \int_{t_k}^{t_{k+1}} a_{ij}(t)dt\leq M
$
for all $k\in\N$ and all $i,j\in[n]$ with $i\neq j$. 
\end{assumption}

Assumption~\ref{assume:uniform_bound} is sufficient to guarantee the strong aperiodicity condition $\Phi(t_{k+1},t_k)\geq \gamma I_n$ for some $\gamma>0$ and all $k\in\N_0$. This is evident from the proof of Lemma 8 in~\cite{martin2015continuous}.

\begin{assumption}[\textbf{Continuous-time Approximate Reciprocity}]\label{assume:approx_reci}  There exist $p_0,\beta\in(0,\infty)$ such that 
$$
    p_0\int_{t_\ell}^{t_m}  \allone_{|S|}^T A_{S\bar S}(t)\allone_{|\bar S|} dt \leq \int_{t_\ell}^{t_m} \allone_{|\bar S|} ^T A_{\bar S S}(t)\allone_{|S|} dt+\beta
$$
holds for all sets $S\subset[n]$ and for all $\ell,m\in\N_0$ with $\ell\leq m$.
\end{assumption}

\begin{lemma}\label{lem:five}
Under Assumption~\ref{assume:uniform_bound}, Assumption~\ref{assume:approx_reci} is equivalent to the existence of constants $\tilde p_0,\tilde\beta\in(0,\infty)$ such that~\eqref{eq:approx_reciprocity_state} holds for all sets $S\subset[n]$.
\end{lemma}
\begin{proof}
We first recall from Proposition 7 of~\cite{martin2015continuous} that under Assumption~\ref{assume:uniform_bound}, there exists a constant $G\in (0,\infty)$ such that 
\begin{align}\label{eq:comparable}
    G   \int_{t_{k}}^{t_{k+1}} \allone_{|S|}^T A_{S\bar S}(t) \allone_{|\bar S|} d t &\leq \allone_{|S|} ^T \phi_{S\bar S}\left(t_{k+1}, t_{k}\right)\allone_{|\bar S|}\cr
    &\leq n\int_{t_{k}}^{t_{k+1}} \allone_{|S|}^T A_{S\bar S}(t) \allone_{|\bar S|} d t\quad\quad
\end{align}
holds for all $k\in\N_0$ and all sets $S\subset[n]$.

Now, suppose Assumption~\ref{assume:approx_reci} holds. Then, for all $S\subset[n]$ and $\ell,m\in\N_0$ with $\ell\leq m$, we have
\begin{align*}
    &G\frac{p_0}{n}\sum_{k=\ell}^m \allone_{|S|}^T\Phi_{S\bar S}(t_{k+1},t_k)\allone_{|\bar S|} \cr
    &\leq G p_0\int_{t_\ell}^{t_{ m+1} } \allone_{|S|} ^T A_{S\bar S}(t) \allone_{|\bar S|} dt\cr
    &\stackrel{(a)}{\leq} G\left(\int_{t_\ell}^{t_{m+1}}\allone_{|S|} ^T A_{S\bar S}(t)\allone_{|\bar S|} dt + \beta\right) \cr
    &\leq \sum_{k=\ell}^m \allone_{|S|} ^T\Phi_{S\bar S}(t_{k+1},t_k)  + G\beta,
\end{align*}
 where $(a)$ follows from Assumption~\ref{assume:approx_reci}. Therefore,~\eqref{eq:approx_reciprocity_state} holds with $\tilde p_0 = \frac{G p_0}{n}$ and $\tilde\beta=G\beta$. 

Similarly, if we are given that~\eqref{eq:approx_reciprocity_state} holds for all $S\subset[n]$, then we can again use~\eqref{eq:comparable} to make arguments similar to the preceding ones to show that Assumption~\ref{assume:approx_reci} holds with $p_0=\frac{G}{n}\tilde p_0$ and $\beta = \frac{\tilde \beta}{n}$. 
\end{proof}
We now use Lemma~\ref{lem:five} to show that approximate reciprocity in continuous time is equivalent to $\{A(k)\}_{k=0}^\infty$ belonging to Class $\pstar$. To begin, we first define the continuous-time analogs of absolute probability sequences and Class $\pstar$.

\begin{definition} [\textbf{Continuous-time Absolute Probability Sequence}~\cite{bolouki2015eminence}] \label{def:ct_abs_prob_seq}
A continuum of stochastic vectors $\{\pi(t)\}_{t\geq 0}$ is said to be an absolute probability sequence for a continuous-time stochastic chain $\{A(t)\}_{t\geq 0}$ if 
\begin{align}\label{eq:last_label}
    \pi^T(t)\Phi(t,\tau)=\pi^T(\tau)
\end{align}
holds for all $t\geq \tau\geq 0$, where $\Phi(\cdot,\cdot)$ denotes the state transition matrix for the dynamics~\eqref{eq:continuous_time_dynamics}.
\end{definition}

\begin{definition} [\textbf{Continuous-time Class $\pstar$}~\cite{bolouki2015eminence}] \label{def:ct_pstar}
We let (Continuous-time Class-)$\mathcal P^*$ be the set of all continuous-time stochastic chains that admit uniformly positive absolute probability sequences, i.e., a continuum of stochastic vectors $\{\pi(t)\}_{t\ge0}$ such that~\eqref{eq:last_label} holds and $\pi(t)\geq p^*\mathbf 1_n$ for some scalar $p^*>0$ and all $t\ge 0$. (Note that the absolute probability sequence and the value of $p^*$ may vary from chain to chain). 
\end{definition}

We are now ready to state the first main result of this section.

\begin{theorem} [\textbf{Continuous-time Analog of Theorem~\ref{thm:main}}]\label{thm:ct_main}
Let $\{A(t)\}_{t\geq 0}$ be a continuous-time stochastic chain that satisfies Assumption~\ref{assume:uniform_bound}. Then $\{A(t)\}_{t\geq 0}$ has a uniformly positive absolute probability sequence if and only if it is approximately reciprocal.
\end{theorem}

\begin{proof}
Suppose $\{A(t)\}_{t\geq 0}$ has a uniformly positive absolute probability sequence, i.e., suppose $\{A(t)\}_{t\geq 0}\in\pstar$. Then we know that $\{\Phi(t_{k+1},t_k)\}_{k=0}^\infty\in\pstar$ in discrete time. It follows from Proposition~\ref{prop:nec_cond} that $\{\Phi(t_{k+1},t_k)\}_{k=0}^\infty$ is approximately reciprocal in discrete time, i.e., there exist constants $\tilde p_0>0$ and $\tilde \beta\in(0,\infty)$ such that ~\eqref{eq:approx_reciprocity_state} holds for all $S\subset[n]$. Lemma~\ref{lem:five} now implies that Assumption~\ref{assume:approx_reci} holds, which means that $\{A(t)\}_{t\geq 0}$ is approximately reciprocal.

On the other hand, suppose we are given that $\{A(t)\}_{t\geq 0}$ is approximately reciprocal with respect to the increasing sequence of times $\{t_k\}_{k=0}^\infty\subset[0,\infty)$. We now show that for any two times $\tau_1, \tau_2\geq 0$ with $\tau_1<\tau_2$, the chain $\{A(t)\}_{t\geq 0}$ is also approximately reciprocal with respect to the augmented sequence of times $t_1,t_2,\ldots,t_q, \tau_1,t_{q+1},\ldots,t_{r},\tau_2,t_{r+1},\ldots$, where $q:=\max\{\ell\in\N_0:t_\ell\leq \tau_1\}$ and ${r:=\min\{\ell\in\N_0:t_\ell\geq \tau_2\}-1}$. Using Assumption~\ref{assume:uniform_bound} for any set $S\subset[n]$, we have
\begin{align*}
    &\int_{\tau_1}^{t_{q+1}} \mathbf 1_{|S|} ^T A_{S\bar S}(t) \mathbf 1 _{|\bar S|}dt\leq \sum_{i\in[n]}\sum_{j\in[n]\setminus\{i\}} \int_{\tau_1}^{t_{q+1}} a_{ij}(t)dt \cr
    &\qquad\leq \sum_{i\in[n]}\sum_{j\in[n]\setminus\{i\}} \int_{t_q}^{t_{q+1}} a_{ij}(t) dt \leq n(n-1)M.
\end{align*}
Similarly, $\int_{t_{q}}^{\tau_1} \allone_{|S|} ^T A_{S\bar S}(t)\allone_{|\bar S|} dt $, $\int_{t_{r}}^{\tau_2} \allone_{|S|} ^T A_{S\bar S}(t)\allone_{|\bar S|} dt $,  and $\int_{\tau_2}^{t_{r+1} } \allone_{|S|} ^T A_{S\bar S}(t)\allone_{|\bar S|} dt$ are all upper bounded by $n(n-1)M$. In addition, we have $\int_{t_\ell}^{t_m} \allone_{|\bar S|} ^T A_{\bar S S}(t)\allone_{|S|} dt \geq 0$ for all $\ell,m\in\N_0$ with $\ell<m$. As a result, the inequality in Assumption~\ref{assume:approx_reci} implies that for all $S\subset[n]$ and $\ell<m$, we have
\begin{align*}
    &p_0\int_{t_\ell'}^{t_m'} \allone_{|S|} ^T A_{S\bar S}(t)\allone _{|\bar S|} dt \cr
    &\leq \int_{t_\ell'}^{t_m'}\allone_{|\bar S|} ^T A_{\bar S S}(t)\allone_{|S|} dt +\beta + 2n(n-1)M p_0,
\end{align*}
where $\{t'_k\}_{k=0}^\infty$ denotes the augmented sequence  $t_1,t_2,\ldots,t_q, \tau_1,t_{q+1},\ldots,t_{r},\tau_2,t_{r+1},\ldots$. Invoking Lemma~\ref{lem:five} now shows that the stochastic chain $\{\Phi(t'_{k+1},t'_k)\}_{k=0}^\infty$ is approximately reciprocal in discrete time. Moreover, Assumption~\ref{assume:uniform_bound} (which continues to hold after replacing $\{t_k\}_{k=0}^\infty$ with $\{t'_k\}_{k=0}^\infty$) and Lemma 8 in~\cite{martin2015continuous} together imply that $\{\Phi(t'_{k+1},t'_k)\}_{k=0}^\infty$ is strongly aperiodic. It now follows from Proposition~\ref{prop:main} that there exists a constant $\eta>0$ such that $\Phi(t'_m:t'_\ell)\geq \eta I_n$ for all $\ell,m\in\N_0$ satisfying $\ell\leq m$. In particular, we have $\Phi(\tau_2:\tau_1)\geq \eta I_n$. As $\tau_1$ and $\tau_2$ are arbitrary, it follows from~\cite[Lemma 8]{bolouki2015eminence} that $\{A(t)\}_{t\geq 0}\in\pstar$.
\end{proof}

The next step is to provide a continuous-time analog of Theorem~\ref{thm:uniqueness}. For this purpose, we define the continuous-time analog of the infinite flow graph as follows.

\begin{definition} [\textbf{Infinite Flow Graph in Continuous Time}] \label{def:ct_infinite_flow_graph} For a continuous-time stochastic chain $\{A(t)\}_{t\geq 0}$, we define its \emph{infinite flow graph} to be the graph $G^\infty= ([n], E^\infty)$ with 
$$
    {E}^{\infty}:=\left\{\{i, j\} \,\Big \lvert\, \int_{0}^{\infty}\left(a_{i j}(t)+a_{j i}(t)\right)dt=\infty, i \neq j \in[m]\right\}.
$$
\end{definition}

We now state the desired theorem.

\begin{theorem} [\textbf{Continuous-time Analog of Theorem~\ref{thm:uniqueness}}] \label{thm:ct_uniqueness}
Let $\{A(t)\}_{t\geq 0}$ be a continuous-time stochastic chain that satisfies Assumptions~\ref{assume:uniform_bound} and~\ref{assume:approx_reci}. Then $\{A(t)\}_{t\geq 0}$ admits a unique absolute probability sequence if and only if its infinite flow graph is connected.
\end{theorem}

\begin{proof}
By repeating some of the arguments used to prove Theorem~\ref{thm:ct_main}, we can show that Assumptions~\ref{assume:uniform_bound} and~\ref{assume:approx_reci} continue to hold (if only with different constants) even if we augment the sequence $\{t_k\}_{k=0}^\infty$ by inserting into it an arbitrary constant $\tau\geq 0$. By Lemma 8 of~\cite{martin2015continuous}, this further implies that the discrete-time chain $\{\Phi(t'_{k+1}:t'_k)\}_{k=0}^\infty$ (where $\{t'_k\}_{k=0}^\infty$ denotes the augmented sequence $t_1,t_2,\ldots,\tau,\ldots,$) is strongly aperiodic. In addition, since $\{A(t)\}_{t\geq 0}$ satisfies the uniform bound assumption (Assumption~\ref{assume:uniform_bound}) in addition to the condition of approximate reciprocity, we know from Theorem~\ref{thm:ct_main} that $\{A(t)\}_{t\geq 0}\in\pstar$. By Definitions~\ref{def:ct_abs_prob_seq} and~\ref{def:ct_pstar}, this implies that $\{\Phi(t'_{k+1}:t'_k )\}_{k=0}^\infty\in\pstar$ in discrete time.  Hence, by Theorem~\ref{thm:main}, $\{\Phi(t'_{k+1}:t'_k )\}_{k=0}^\infty$ is approximately reciprocal.

Now, the infinite flow graph of $\{A(t)\}_{t\geq 0}$ being connected is equivalent to $\int_{0}^\infty \allone_{|S|} ^T A_{S\bar S}(t)\allone_{|\bar S|} dt+\int_{0}^\infty \allone_{|\bar S|} ^T A_{\bar S S}(t)\allone_{|S|} dt=\infty$ being satisfied for all $S\subset[n]$, which, by ~\cite[Proposition 7] {martin2015continuous}, is in turn equivalent to the infinite flow graph of the chain $\{\Phi(t'_{k+1}:t'_k)\}_{k=0}^\infty$ being connected. By the strong aperiodicity and the approximate reciprocity of  $\{\Phi(t'_{k+1},t'_k)\}_{k=0}^\infty$ (shown above), Theorem~\ref{thm:uniqueness} implies that the connectivity of the infinite flow graph of $\{\Phi(t'_{k+1},t'_k)\}_{k=0}^\infty$ is equivalent to {$\{\Phi(t'_{k+1},t'_k)\}_{k=0}^\infty$ } having a unique absolute probability sequence. 

{Thus}, the infinite flow graph of $\{A(t)\}_{t\geq 0}$ is connected if and only if $\{\Phi(t'_{k+1}:t'_k)\}_{k=0}^\infty$ has a unique absolute probability sequence, i.e., if and only if the  vectors $\{\pi(t_k)\}_{k=0}^\infty\cup\{\pi(\tau)\}$ are unique. Since $\tau$ is arbitrary, it follows that the infinite flow graph of $\{A(t)\}_{t\geq 0}$ is connected if and only if the absolute probability sequence $\{\pi(\tau)\}_{\tau\geq 0}$ is unique.
\end{proof}

\section{Applications}\label{sec:applications}

We now derive a few corollaries of our main results. {It is worth noting that many of these corollaries have been hitherto known to hold only for instantaneously reciprocal chains and not for the broader class of approximately reciprocal chains.}
\subsection{Infinite Flow Stability of Independent Random Chains}
The concept of independent random chains is a straightforward extension of that of deterministic chains: a discrete-time stochastic chain $\{A(k)\}_{k=0}^\infty$ is called an independent random chain if $\{A(k)\}_{k=0}^\infty$ are all random and independently distributed. Note that every deterministic  chain is an independent random chain composed of degenerate random matrices.

Based on this definition, we can extend the notion of Class $\pstar$ to independent random chains as follows:
an independent random chain $\{A(k)\}_{k=0}^\infty$ is said to belong to Class $\pstar$ if the expected chain $\{\E[A(k)]\}_{k=0}^\infty$ belongs to Class $\pstar$.

For an application of our main results to independent random chains, we will also need a notion of strong aperiodicity for such chains. We introduce this notion as follows: suppose $\{A(k)\}_{k=0}^\infty$ is an independent random chain. We say that $\{A(k)\}_{k=0}^\infty$ has the \textit{feedback property}~\cite{touri2013product} if there exists a {\textit{feedback coefficient}} $\gamma>0$ such that $\E[a_{ii}(k)a_{ij}(k)]\geq \gamma \E[a_{ij}(k)]$ for all $k\in\N_0$ and all distinct $i,j\in[n]$.

In addition to the feedback property, we need a concept that captures the notion of ergodicity (Definition~\ref{def:ergodicity}) for pairs of row indices of a stochastic chain. Consider a {random} stochastic chain $\{A(k)\}_{k=0}^\infty$. We say that $i\in[n]$ and $j\in[n]$ are \textit{mutually ergodic} indices for $\{A(k)\}_{k=0}^\infty$, which we denote by $i\leftrightarrow_A j$, if $\lim_{k\rightarrow\infty} (x_i(k) - x_j(k)) =0$ {holds \textit{a.s.}}\ for the dynamics $x(k+1) = A(k)x(k)$ started with an arbitrary initial condition $x(k_0) = x_0$ (where $k_0\in\N_0$ and $x_0\in\R^n$). {If $\{A(k)\}_{k=0}^\infty$ is deterministic, we adopt the same definition for mutual ergodicity after dropping the qualifier ``almost surely''.}

Based on these concepts, we have the following result.

\begin{corollary}\label{cor:inf_flow_stable} Let $\{A(k)\}_{k=0}^\infty$ be an independent random chain with feedback property, and suppose the expected chain $\{\bar A(k)\}_{k=0}^\infty:=\{\E[A(k)]\}_{k=0}^\infty$ is approximately reciprocal. Then,
\begin{enumerate} [(i)]
    \item \label{item:inf_flow_stable} $\{A(k)\}_{k=0}^\infty$ is infinite flow stable almost surely. 
    \item \label{item:mutual_ergodic} For any two indices $i$ and $j$ in $[n]$, we have $i\leftrightarrow_A j$ if and only if $i\leftrightarrow_{\bar A} j$.
    \item\label{item:same_connected_comp} $i$ and $j$ belong to the same connected component of $G^\infty$ if and only if $i$ and $j$ belong to the same connected component of $\bar G^\infty$, the infinite flow graph of $\{\bar A(k)\}_{k=0}^\infty$.
\end{enumerate}
\end{corollary}

\begin{proof} 
$\{A(k)\}_{k=0}^\infty$ having feedback property implies that $\{\E[A(k)]\}_{k=0}^\infty$ is strongly aperiodic~\cite{touri2013product}. Since $\{\E[A(k)]\}_{k=0}^\infty$ is also approximately reciprocal, we know from Theorem~\ref{thm:main} that $\{\E[A(k)]\}_{k=0}^\infty\in\pstar$. {Equivalently,} $\{A(k)\}_{k=0}^\infty\in \pstar$ {by our definition of Class $\pstar$ for independent random chains}. Assertion~\ref{item:inf_flow_stable} now follows from  ~\cite[Theorem 2]{touri2011existence} and the remaining assertions follow from ~\cite[Theorem 5.1]{touri2012product}.
\end{proof}

\begin{remark}\label{remark:new_remark}
By Definition~\ref{def:inf_flow_stab}, Assertion~\ref{item:inf_flow_stable} of Corollary~\ref{cor:inf_flow_stable} implies the following: if $\{A(k)\}_{k=0}^\infty$ is either {(a)} a {strongly aperiodic and approximately reciprocal} deterministic chain, or {(b)} an independent random chain with feedback property {such that $\{\E[A(k)]\}_{k=0}^\infty$ is approximately reciprocal}, then $\lim_{k\rightarrow\infty}A(k:k_0)$ exists (\textit{a.s.}) for all $k_0\in\N_0$. In light of Assertions~\ref{item:mutual_ergodic} and~\ref{item:same_connected_comp}, this further implies that, for any two indices $i,j\in[n]$ and an arbitrary starting time $k_0\in\N_0$, the event that the $i$-th row of $\lim_{k\to\infty} A(k:k_0)$ equals the $j$-th row of $\lim_{k\to\infty} A(k:k_0)$ almost surely equals the event that $i$ and $j$ belong to the same connected component of {$ G^\infty$}.
\end{remark}

\subsection{Rate of Convergence to Steady State}

We now provide a result on the rate of convergence for the dynamics $x(k+1) = A(k)x(k)$ in terms of the quadratic comparison function $V_u(x)=\sum_{i=1}^{m} u_{i}\left(x_{i}-u^{T} x\right)^{2}$, where $u$ is an arbitrary stochastic vector in $\R^n$.

\begin{proposition}
\label{prop:conv_rate}
    Let $\{A(k)\}_{k=0}^\infty$ be an independent random chain with feedback property and feedback coefficient $\gamma>0$, and suppose the expected chain $\{\bar A(k)\}_{k=0}^\infty$ is approximately reciprocal. In addition, let $k_q=0$ for $q=0$ and let
$$
    k_{q}=\underset{t \geq k_{q-1}+1}{\operatorname{argmin}} \operatorname{Pr}\left(\min _{S \subset[n]} \sum_{t=k_{q-1}}^{t-1} \allone_n^T A_{S}(k) \allone_n \geq \delta\right) \geq \epsilon
$$
for all $q\geq 1$. Then, for all $q\geq 1$ and all stochastic vectors $u\in\R^n$,
$$
\E\left[V_u\left(x\left(k_{q}\right), k_{q}\right)\right] \leq\left(1-\frac{\epsilon \delta(1-\delta)^{2} \gamma p^{*}}{(m-1)^{2}}\right)^{q} \E[V_u(x(0), 0)].
$$
\end{proposition}

\begin{proof}
We can repeat the arguments used in the proof of Corollary~\ref{cor:inf_flow_stable} to show that $\{\E[A(t)]\}_{t=0}^\infty\in\pstar$. Therefore, this result is a straightforward consequence of Theorem~\ref{thm:main} above, Theorem 5.2 of~\cite{touri2012product}, and the implication that $\{\E[A(t)]\}_{t=0}^\infty$ has the strong feedback property.
\end{proof}

\subsection{Implications for Sonin's Jet Decomposition}

For a stochastic chain to be ergodic, it is necessary for the chain to possess a property called the \textit{infinite jet-flow} property~\cite{bolouki2015consensus}. In this subsection, our aim is to connect the concept of approximate reciprocity with the infinite jet-flow property and also with  the related concept of Sonin's jet decomposition~\cite{bolouki2015consensus,sonin2008decomposition}{, which we introduce in Proposition~\ref{prop:sonin}}. We first define the concept of \textit{jets}, first introduced in~\cite{bolouki2015consensus}.
 
 For any set $S\subset[n]$, we say that a sequence of sets $\{J(k)\}_{k=0}^\infty$ that satisfies $J(k)\subset S$ for all $k\in\N_0$ is a \textit{jet} in $S$. On the basis of this, we say that a tuple of jets $(J_1(k),J_2(k),\ldots, J_q(k) )$ is a jet-partition of $[n]$ if $\bigcup_{\ell\in[q] }J_\ell(k)=[n]$ and $J_r(k)\cap J_s(k)=\emptyset$ for all $r\ne s$ and $k\in\N_0$. In addition, for a jet $J$, the \textit{jet limit} $J^*$ denotes $\lim_{k\to\infty} J(k)$ if
it exists, in the sense that the set $J(k)$ becomes constant after a finite period of time. 


We now have the following result.
\def\Iess{I_{\text{ess}}}

\begin{proposition}\label{prop:sonin}
Let $\{A(k)\}_{k=0}^\infty$ be a strongly aperiodic and an approximately reciprocal stochastic chain, and let its infinite flow graph $G^\infty$ have $c$ connected components with vertex sets $\{J_\ell^*:\ell\in[c]\}$. Then $\{J_\ell^*:\ell\in[c]\}$ constitute the jet limits in Sonin's jet decomposition~\cite{sonin2008decomposition} of $\{A(k)\}_{k=0}^\infty$. {Equivalently,} there exists a jet-partition $(J_0(k), J_1(k),\ldots, J_c(k))$ of $[n]$ such that the following assertions hold:
\begin{enumerate}
    \item $\{J_\ell^*:\ell\in[c]\}$ are the jet limits of $\{J_\ell(k):\ell\in[c]\}$.
    \item For every $\ell\in [c]$, there exists a constant $\pi_\ell^*$ {such that}
    $
        {\lim _{k \rightarrow \infty} \sum_{i \in J_\ell(k)} \pi_i(k)=\pi_\ell^*.}
    $
    \item For every $\ell\in [c]$ and $(k_0,x_0)\in \N_0\times\R^n$, there exists a {scalar}  $x_\ell^*(k_0,x_0)\in\R$ such that 
    ${
        \lim_{k\to\infty} \left(A(k:k_0)x_0 \right)_{i_k}=x_\ell^*(k_0,x_0)}
    $
    for every sequence $\{i_k\}_{k=0}^\infty$ such that $ i_k\in J_\ell(k)$ for all ${k\in\N_0}$. 
    \item The total \emph{flow} between any two jets is finite, i.e.,
    \begin{align*}
        \sum_{k=0}^\infty\bigg[&\sum_{i \in J^r(k)} \sum_{j \in J^s(k+1)} \pi_j(k+1)a_{ji}(k)\cr
        &+\sum_{i \in J^r(k)} \sum_{j \in J^s(k+1)} \pi_j(k+1)a_{ji}(k)\bigg]<\infty
    \end{align*}
    for all distinct $r,s\in[c]\cup\{0\}$.
    \item $\lim _{k \rightarrow \infty} \sum_{i \in J^0(k)} \pi_i(k)=0$.
\end{enumerate}
\end{proposition}

\begin{proof}
The fact that $\{J_k:k\in[c]\}$ constitute the jet limits in Sonin's decomposition follows  from Theorem 1, Corollary 3,  ~\cite[Theorem 4] {bolouki2015consensus}, and from the fact that strong aperiodicity implies weak aperiodicity. Assertions 1 - 5 now follow from {Sonin's definition of jet decomposition}~\cite[Theorem 1] {sonin2008decomposition}.
\end{proof}

\subsection{Generalized Deffuant-Weisbuch Dynamics}

So far, we have analyzed state-independent dynamics, i.e., dynamics for which the state evolution matrix $A$ (or its expectation $\bar A$) is a function only of the time $k$ and not of the state $x(k)$. To show that our main results can also be applied to state-dependent dynamics, 
we now consider a generalization of \textit{Deffuant-Weisbuch} dynamics~\cite{deffuant2001mixing}, a model of opinion dynamics that incorporates \textit{bounded confidence}, which is the notion that individuals in a social network influence each other's opinions only if they are similarly opinionated.

To this end, consider a social network of $n$ individuals with arbitrary initial opinions $\{x_i(0)\}_{i=1}^n$. At each time step, a pair of {distinct} agents $i,j \in [n]$  is chosen randomly with a constant probability $q_{ij}>0$ (and hence, $\sum_{1\le i<j\le n }q_{ij}=1$). The agents update their opinions if and only if the difference between their current opinions is no greater than a constant confidence threshold {$r_{ij}=r_{ji}>0$}.  {Precisely, the agent pair chosen at time $k$, which we denote by $(\ell(k),m(k))$,} update their opinions as 
\begin{align}\label{eq:deffuant}
    x_{\ell(k)} (k+1) &= \alpha(k)x_{\ell(k)} (k) + (1-\alpha(k))x_{m(k)} (k),\cr
    x_{m(k)} (k+1) &= (1-\beta(k))x_{\ell(k)} (k) + \beta(k)x_{m(k)} (k)
\end{align}
if {$|x_{\ell(k)} (k)-x_{m(k)} (k)|\leq r_{\ell(k)\, m(k) }$}, where $\alpha(k), \beta(k)\in (0,1)$ are random self-weights, and the agents' opinions remain the same if {$|x_{\ell(k)} (k)-x_{m(k)} (k)|> r_{\ell(k)\, m(k) }$}. This model is  more general than the classical Deffuant-Weisbuch model, which assumes that $i$ and $j$ are chosen \textit{uniformly} at random  and that there exists a constant $\mu>0$ such that $\alpha_i(k)=\mu$ for all $i\in[n]$ and $k\in\N_0$.

We now provide a convergence result for the above model. This result applies to the scenario in which all the agents' self-weights are almost surely ``moderate'' in that they are {(i) always above a constant positive threshold, and (ii)} uniformly bounded away from $1$ throughout any given time interval, except possibly for a sub-interval of bounded duration.

\begin{proposition}\label{prop:deffuant_weisbuch} Consider the generalized Deffuant-Weisbuch dynamics defined above. Suppose there exists a constant ${\delta\in (0,1)}$ such that  { {$\min\{\alpha(k), \beta(k)\}\ge \delta$ \textit{a.s.}} for all $k\in\N_0$. In addition, suppose there exist constants $\varepsilon\in (0,1-\delta)$} and $B<\infty$ such that for any two times $k_1,k_2\in\N_0$ with $k_2-k_1>B$, there exists a set of times $I\subset \{k_1,\ldots, k_2\}$ with $|I|\geq k_2-k_1-B$ such that $\max\{\alpha(k),\beta(k)\}\le 1-\varepsilon$ \textit{a.s.} for all $k\in I$. Then $\lim_{k\to\infty} x(k)$ exists \textit{a.s.} for all initial conditions $(k_0, x(k_0))\in \N_0\times \R^n$.
\end{proposition}
{
The proof of Proposition~\ref{prop:deffuant_weisbuch} is relegated to Appendix~\ref{app:C}.}

\subsection{Some Other Applications}

Below we discuss a few other applications of our  results.
\begin{enumerate}
    \item \textit{Multiple Consensus:} We say that {multiple consensus}~\cite{bolouki2014geometric}  occurs whenever $\lim_{t\to\infty} x(t)$ exists but is not necessarily a multiple of the consensus vector $\allone_n$, meaning that different entries of $x(t)$ may or may not converge to different limits. An immediate consequence of Theorem~\ref{thm:ct_main} above and Theorem 2 of~\cite{bolouki2014geometric} is that multiple consensus always occurs in the continuous-time dynamics $\dot x(t)=A(t)x(t)$ if  $\{A(t)\}_{t\geq 0}$ is an approximately reciprocal chain that satisfies Assumption~\ref{assume:uniform_bound}.
    \item \textit{\'Eminence Grise Coalitions:} An \'eminence grise coalition (EGC,~\cite{bolouki2015eminence}) is a subset of the total agent population that has the ability to steer the opinions of all the individuals in the network to a desired consensus asymptotically in time. A direct consequence of Theorem~\ref{thm:ct_main} above and Corollary 3 of~\cite{bolouki2015eminence} is as follows: if $\{A(t)\}_{t\geq 0}$ is an approximately reciprocal chain satisfying Assumption~\ref{assume:uniform_bound}, then the size of a minimal EGC of a network with dynamics $\dot x(t)=A(t)x(t)$ is the number of connected components in the infinite flow graph of $\{A(t)\}_{t\geq 0}$.
    \item \textit{Distributed Optimization:} A typical distributed optimization framework consists of a network of $n$ interacting agents with the common objective of minimizing the sum of $n$ convex functions $\{f_i:\R^d\rightarrow\R^d \}_{i=1}^n$ subject to the constraint that for each $i\in[n]$, the function $f_i$ is known only to agent $i$. Notably,~\cite{touri2015continuous} provides a continuous-time algorithm for distributed optimization without requiring the associated stochastic chain $\{A(t)\}_{t\geq 0}$ to be \textit{cut-balanced}~\cite{touri2013product}. However, the results therein are based on an assumption involving an abstract concept called \textit{Class $\pstar$ flows}, the interpretation of which is aided significantly by results such as  Theorem~\ref{thm:ct_main}.
    \item \textit{Distributed Learning/Hypothesis Testing:} In a typical distributed learning scenario, there is a set of possible states of the world, of which a subset of states are true. In addition, there is a network of interacting agents whose common objective is to learn the identity of the true state through mutual interaction as well as by performing private measurements on the state of the world. Our prior work~\cite{parasnis2022non} generalizes certain known results on distributed learning to networks described by random, independently distributed time-varying directed graphs. Importantly, the sequence of weighted adjacency matrices of all the networks considered therein are assumed to belong to Class $\pstar$. Hence,  along with the concept of Class $\pstar$ for independent random chains, Theorem~\ref{thm:main} significantly facilitates our interpretation of the main results of~\cite{parasnis2022non}.
\end{enumerate}

\section{Conclusion}\label{sec:conclusion}

We extended two eigenvector assertions of the classical Perron-Frobenius theorem to sequences as well as continua of row-stochastic matrices that satisfy the mild assumption of strong aperiodicity. In the process, we established approximate reciprocity as an equivalent characterization of Class $\pstar$, a special but broad class of stochastic chains that subsumes a few important sub-classes such as cut-balanced (instantaneously reciprocal) chains, doubly stochastic chains, and uniformly strongly connected chains~\cite{touri2013product}. We then discussed a few applications of our main results to problems in distributed learning, distributed averaging, opinion dynamics, etc.

{Exploring the connections between Theorems~\ref{thm:main} -~\ref{thm:ct_uniqueness} and other extensions of the Perron-Frobenius theorem, in particular the Krein-Rutman theorem~\cite{krein1948linear}, is a very interesting direction for future research.} 
We would also like to extend our results to dependent random chains in order to  study  random real-world phenomena. Finally, we will attempt to extend our results to sequences of non-negative matrices that are not necessarily row-stochastic, as this would result in a complete generalization of the eigenvector assertions of the Perron-Frobenius theorem to  time-varying matrices.

Nevertheless, we {expect our results} in their present form to find a significant number of applications other than those discussed above. This belief is rooted in the already wide applicability of the classical theorem.

\bibliography{bib}
\bibliographystyle{ieeetr}

\section*{Appendices}

\section{Auxiliary Lemmas}
\begin{lemma}\label{lem:exponential} Let $\varepsilon\in (0,1)$ be given. Then $1-x\geq e^{-M(\varepsilon)x}$ for all $x\in [0,1-\varepsilon]$, where $M(\varepsilon):= \frac{1}{1-\varepsilon}\ln\frac{1}{\varepsilon}$.
\end{lemma}

\begin{proof}
Let $f:[0, 1-\varepsilon]\rightarrow\R$ be defined by $f(x)=1-x-e^{-M(\varepsilon)x}$. Then $f(0)=0$. Next, note that $f''(x)=-M(\varepsilon)^2e^{-M(\varepsilon) x}<0$ for all $x\in[0,1-\varepsilon]$, implying that $f$ is concave on its domain. Also, observe that $f(1-\varepsilon)=0$. Therefore, by Jensen's inequality, for any $x\in[0, 1-\varepsilon]$, we have $f(x)=f\left(\frac{x}{1-\varepsilon}(1-\varepsilon)+\left(1-\frac{x}{1-\varepsilon}\right)\cdot 0\right)
    \geq \frac{x}{1-\varepsilon} f(1-\varepsilon) + \left(1-\frac{x}{1-\varepsilon}\right)f(0)=0.$
\end{proof}

\begin{lemma}\label{lem:obv} Suppose $G^\infty=([n],E^\infty)$, the infinite flow graph of $\{A(k)\}_{k=0}^\infty$, is connected. Then $\{A(k)\}_{k=0}^\infty$ is ergodic if it is infinite flow stable.
\end{lemma}

\begin{proof}
As $G^\infty$ is connected, for all $i,j\in[n]$  there exists a path between $i$ and $j$ in $G^\infty$, i.e., there exists an $r\in[n]$ and vertices $\ell_1,\ell_2,\ldots, \ell_r\in [n]$ with $\ell_1=i$ and $\ell_r=j$ such that $(\ell_1,\ell_2), (\ell_2,\ell_3) \ldots, (\ell_{r-1},\ell_r)\in E^\infty$. As $\{A(k)\}_{k=0}^\infty$ is also infinite flow stable, this implies $\lim_{k\rightarrow\infty} (x_{\ell_k}(k)-x_{\ell_{k+1}}(k)) = 0$ for all $k\in[r-1]$. As a result, 
$\lim_{k\to\infty}(x_i(k)-x_j(k))=0$. As $i$ and $j$ are arbitrary, it follows from~\cite[Theorem 2.2]{touri2012product} that $\{A(k)\}_{k=0}^\infty$ is ergodic.
\end{proof}
\section{Proof of Lemma~\ref{lem:uncertain}}\label{app:A}
\begin{proof}
We define $N=|\{k\in\{0,\ldots,\sigma-1\}:B_{ji}(k)>0\}|$ and use induction on $N$. For $N=1$, we have $B_{ji}(k_L)\geq \delta$ and hence the following, which verifies the lemma:
\begin{align} \label{eq:re-referenced}
     B_{ji}(\sigma:0)\geq B_{jj}(\sigma:k_L+1)B_{ji}(k_L)B_{ii}(k_L:0)\geq \eta_j \delta\eta_i.
\end{align}

Now, suppose the lemma holds when $N=N_0$ for some $N_0\in\N$, and consider $N=N_0+1$. We define $\varepsilon:=B_{ji}(k_L)$, and consider two cases. If $\varepsilon\geq \delta$, i.e., $B_{ji}(k_L)\geq \delta$, then \eqref{eq:re-referenced} still holds, thereby proving the lemma. On the other hand, if $\varepsilon<\delta $, then we let $\tilde B(k):=B(k)$ for each $k\in\{0,\ldots,\sigma-1\}\setminus\{k_L\}$, and $\tilde{B}(k_L):=B(k_L) - B_{ji}(k_L) e_je^T_i$. 
Therefore, $\{\tilde B(k)\}_{k=0}^{\sigma-1}$ is a sequence of substochastic matrices satisfying $|\{k\in\{0,\ldots,\sigma-1\}:\tilde B_{ji}(k)>0\}|=N_0$. 

Next, we have $\tilde B_{ii}(k_1:k_0)=B_{ii}(k_1:k_0)\geq \eta_i\text{ whenever } 0\leq k_0\leq k_1\leq k_L$. Since the definitions of $k_L$ and $\{\tilde B(k)\}_{k=0}^{\sigma-1}$ imply that $\tilde k_L:=\max\{k\leq \sigma-1 :\tilde B_{{ji}}(k)>0\}<k_L$, it follows that $\tilde B_{ii}(k_1:k_0)\geq \eta_i$ whenever $0\leq k_0\leq k_1\leq \tilde k_L$. Next, note that for all $k_0, k_1$ satisfying $0\leq k_0\leq k_1\leq \sigma$ and $\{k_0, \ldots, k_1-1\}\not\ni k_L$, we have $\tilde B_{jj}(k_1:k_0)=B_{jj}(k_1:k_0)\geq \eta_j$ whereas for all $k_0,k_1$ satisfying $0\leq k_0\leq k_L< k_1\leq \sigma$, we have $\tilde B_{jj}(k_1:k_0) 
    = B_{jj}(k_1:k_0) - B_{jj}(k_1:k_L+1)B_{ji}(k_L)B_{ij}(k_L:k_0)
    \geq \eta_j - \varepsilon$
because the substochasticity of $\{B(k)\}$ implies that $\max\{B_{jj}(k_1:k_L+1), B_{ij}(k_L:k_0)\}\leq 1$. Moreover, $\sum_{k=0}^{\sigma -1}\tilde B_{ji}(k) = \sum_{k=0}^{\sigma -1}\tilde B_{ji}(k) - B_{ij}(k_L) \geq \delta - \varepsilon$. Thus, \( \tilde B_{ii}(k_1:k_0) \geq \eta_i \) if \( 0 \leq k_0 \leq k_1 \leq \tilde k_L \), \( \tilde B_{jj}(k_1:k_0) \geq \tilde \eta_j \) if \( 0 \leq k_0 \leq k_1 \leq \sigma \), and \( \sum_{k=0}^{\sigma-1} B_{ji}(k) \geq \tilde \delta \), where $\tilde \eta_j:=\eta_j - \varepsilon > \delta - \varepsilon > 0$ and $\tilde \delta :=\delta - \varepsilon \in (0, \tilde\eta_j)$. Therefore, by our inductive hypothesis, we have $\tilde B_{ji}(\sigma:0)\geq \frac{1}{2}\eta_i \tilde \eta_j \tilde \delta = \frac{1}{2} \eta_i (\eta_j - \varepsilon) (\delta-\varepsilon) $.

Now, observe that
\begin{align*}
    &B_{ji}(\sigma:0)\cr
    &= \tilde B_{ji}(\sigma:0) + B_{jj}(\sigma:k_L+1)B_{ji}(k_L)B_{ii}(k_L:0)\cr
    &\geq \frac{1}{2} \eta_i (\eta_j - \varepsilon) (\delta-\varepsilon) + \eta_j \varepsilon \eta_i\cr
    &=\frac{1}{2}\eta_i\varepsilon^2 + \frac{1}{2}\eta_i (\eta_j - \delta) \varepsilon + \frac{1}{2}\eta_i\eta_j\delta \stackrel{(a)}\geq \frac{1}{2}\eta_i\eta_j\delta,
\end{align*}
where $(a)$ holds because $\varepsilon>0$ and $\eta_j>\delta$. The lemma thus holds for $N=N_0+1$ and hence, for all $N \leq \sigma$.
\end{proof}

\section{Proof of Proposition~\ref{prop:main}}\label{app:B}
\begin{proof}
We use induction on $n$, the matrix dimension. Consider $n=1$, suppose that $\gamma,p_0\in(0,1),\beta\in(0,\infty)$ and $\Delta\in[0,\infty)$ are given, and let $\{A(k)\}_{k=0}^\infty=\{a(k)\}_{k=0}^\infty$ be a sequence of real numbers satisfying the three properties required by the proposition. Then, by the feedback property of the chain, $\{a_k\}_{k=0}^\infty$ is a sequence of scalars in $[\gamma,1]$. Let $\bar a_k:=1-a_k$ for each $k\in\N_0$. Then $\bar a_k\in [0, 1-\gamma]$ for all $k\in\N_0$, and
$
\sum_{k=0}^\infty
\bar a_k\leq \Delta$ by approximate stochasticity. Hence, for any given $k_0,k_1\in \N_0$ satisfying $k_0\leq k_1$,
\begin{align*}
    &A(k_1:k_0)=\prod_{k=k_0}^{k_1-1}(1-\bar a_k)
    \stackrel{(a)}{\geq} \prod_{k=k_0}^{k_1-1} e^{-M(\gamma)\bar a_k}\cr
    &=e^{-M(\gamma)\sum_{k=k_0}^{k_1-1}\bar a_k}
    \geq e^{-M(\gamma)\sum_{k=0}^{\infty}\bar a_k}\geq e^{-M(\gamma)\Delta}>0,
\end{align*}
where (a) is a consequence of Lemma \ref{lem:exponential}. Thus, we may set $\eta_1(\gamma,p_0,\beta,\Delta)=e^{-M(\gamma)\Delta}$. {As $M$ is a continuous function}, this proves the proposition for $n=1$.

Now, suppose the proposition holds for all $n\leq q$ for some $q\geq 1$, and consider $n=q+1$. We again suppose that $\gamma,p_0,\beta$ and $\Delta$ are given, and let $\{A(k)\}_{k=0}^\infty$ be a substochastic chain in $\R^{n\times n}$ satisfying the required properties. For each $k\in\N_0$, let $v(k):=\allone_n-A(k)\allone_n$ and $v_{\max}(k):=\max_{i\in [n]} (v(k))_i$. Observe that the feedback property and the substochasticity of $A(k)$ together imply that $\mathbf 0_n\leq v(k)\leq (1-\gamma)\allone_n$  for all $k\in\N_0$. We also observe that $A(k)\allone_n\geq (1-v_{\max}(k))\allone_n$ for all $k\in\N_0$. Therefore, for all $0\leq k_0\leq k_1<\infty$, we have
\begin{align}\label{eq:matrix_prod_ineq}
    & A(k_1:k_0)\allone_n=A(k_1-1)\cdots A(k_0+1) A(k_0)\allone_n\cr
    &\geq A(k_1-1)\cdots A(k_0+1)(1-v_{\max}(k_0))\allone_n\cr &\stackrel{(a)}{\geq} \left(\prod_{k=k_0}^{k_1-1}(1-v_{\max}(k) )\right)\allone_n\stackrel{(b)}{\geq} e^{-M(\gamma)\sum_{k=k_0}^{k_1-1}v_{\max}(k) }\allone_n\cr
    &\geq e^{-M(\gamma)\sum_{k=k_0}^{k_1-1}\allone_n^T v(k)}{\allone_n}\stackrel{(c)}{\geq} e^{-M(\gamma)\Delta}\allone_n,
\end{align}
where (a) can be easily shown by induction, (b) is obtained by a repeated application of Lemma \ref{lem:exponential}, and (c) follows from the approximate stochasticity of the chain. 

We now construct two chains of substochastic matrices with dimensions smaller than $n$ and apply our inductive hypothesis to the resulting chains. Let $\{\tau_0,\tau_1,\tau_2,\ldots, \tau_n\}\subset\N\cup\{\infty\}$ be the times defined by $\tau_0:=k_0$ and
$
    \tau_l:=\inf\left\{\tau\geq \tau_{l-1}:\min_{T\subset[n]}\sum_{k=\tau_{l-1}}^{\tau-1}\allone_{|T|} ^T A_{T\bar T}(k)\allone_{|\bar T|}  >  1\right\},
$
{so that $\{ \{\tau_{\ell-1},\ldots,\tau_{\ell}\}:\ell\in [n]\}$ are the shortest consecutive intervals over which the influence of any subset $T\subset [n]$ on the complementary set $\bar T$ exceeds a fixed positive threshold (chosen to be $1$ for simplicity). As we show later in this proof, each of these $n$ intervals corresponds to an $n\times n$ irreducible matrix with positive diagonal entries. The product of $n-1$ or more such matrices is positive~\cite{4570093} -- a  fact that we will use to show that the backward matrix product of $\{A(k)\}_{k=0}^\infty$ over $[\tau_0, \tau_n]$ is entry-wise lower-bounded by a positive matrix. }

Now, let $m=\max\{s:\tau_s<\infty\}$ so that $\tau_s=\infty $ if and only if $s>m$, and consider any $s\in\{0,1,\ldots, \min\{m, n-1\} \}$. Then, by the definition of $\tau_{s+1}$, there exists at least one set $T\subset [n]$ such that $\sum_{k=\tau_s}^{\tau_{s+1}-2} \allone_n^T A_{T \bar T}(k)\allone_n \leq 1$ (note that this also holds if $m\leq n-1$ and $s=m$, in which case $\tau_{s+1}-2=\infty$). We choose any one such set $T$ and assume that $T=[|T|]$ w.lo.g. We accordingly define the chains $\{B(k)\}_{k=\tau_s}^\infty$ and $\{C(k)\}_{k=\tau_s}^\infty$ as 
\begin{gather*}
    B(k) =
    \begin{cases}
    A_T(k) &\, \text{if } \tau_s\leq k\leq \tau_{s+1}-1,\\
    I_{|T|} &\,\text{otherwise},
    \end{cases}\\
    C(k) =
    \begin{cases}
    A_{\bar T}(k) \, \text{if } \tau_s\leq k\leq \tau_{s+1}-1,\\
    I_{|\bar T|} \,\text{otherwise}.
    \end{cases}
\end{gather*}
Now, the definition of $T$ implies  $\sum_{k=\tau_s}^{\tau_{s+1}-1}\allone_{|T|}^T A_{T \bar T}(k)\allone_{|\bar T|} \leq 1+n\leq 2n$. Due to approximate reciprocity, it follows that $\sum_{k=\tau_s}^{\tau_{s+1}-1}\allone_{|\bar T|}^T A_{\bar T T}(k)\allone_{|T|} \leq \frac{2n+\beta}{p_0}$. Note that the inequality $\sum_{k=\tau_s}^{\tau_{s+1}-1}\allone_{|T|}^T A_{T\bar T}(k)\allone_{|\bar T|} \leq 2n$ also implies that $\sum_{k=\tau_s}^{\tau_{s+1}-1} \allone_n^T(\allone_{| T|}-A_T(k)\allone_{|T|})
    =  \sum_{k=\tau_s}^{\tau_{s+1}-1} \allone_{|T|}^T (A_{ T \bar T}(k)\allone_{|\bar T|} + v_T(k)) \leq 2n + \Delta.$
Similarly, $\sum_{k=\tau_s}^{\tau_{s+1}-1}\allone_{|\bar T|} ^T A_{\bar T T}(k)\allone_{|T|} \leq \frac{2n+\beta}{p_0}$ implies that
${\sum_{k=\tau_s}^{\tau_{s+1}-1} \allone_{|\bar T|}^T(\allone_n-A_{\bar T}(k)\allone_{|\bar T|})\leq \frac{2n+\beta}{p_0}+\Delta} 
$.
Therefore, $\{A_T(k)\}_{k=\tau_s}^{\tau_{s+1}-1}$ and $\{A_{\bar T}(k)\}_{k=\tau_s}^{\tau_{s+1}-1}$ are both approximately stochastic sequences. It follows that $\{B(k)\}_{k=\tau_s}^\infty$ and $\{C(k)\}_{k=\tau_s}^\infty$ are also approximately stochastic. 

Next, for any subset $U\subset T$, let $\bar U:=[n]\setminus U$ and {${ U^c:=T\setminus U}$}. Then $\{A(k)\}_{k=0}^\infty$, being approximately reciprocal, satisfies
\begin{align*}
& p_0 \sum_{k=k_0}^{k_1-1} \mathbf{1 }_{| U^c| }^T A_{U^c U}(k) \mathbf{1 }_{| U| } \leq p_0 \sum_{k=k_0}^{k_1-1} \mathbf{1 }_{| \bar U| }^T A_{\bar{U} U}(k) \mathbf{1 }_{| U| } \\
& \leq \sum_{k=k_0}^{k_1-1} \mathbf{1 }_{| U| }^T A_{U \bar{U}}(k) \mathbf{1 }_{| \bar U| }+\beta \\
& =\sum_{k=k_0}^{k_1-1} \mathbf{1}_{| U|} ^T A_{U {U}^c}(k) \mathbf{1}_{| U^c| } +\sum_{k=k_0}^{k_1-1} \mathbf{1 }_{| U| }^T A_{U \bar{T}}(k) \mathbf{1 }_{| \bar T| }+\beta \\
& \leq \sum_{k=k_0}^{k_1-1} \mathbf{1 }_{| U^c| }^T A_{U U^c}(k) \mathbf{1}_{| U^c| } +\sum_{k=k_0}^{k_1-1} \mathbf{1}_{| T| }^T A_{T \bar{T}}(k) \mathbf{1}_{| \bar T| }+\beta \\
& \leq \sum_{k=k_0}^{k_1-1} \mathbf{1}_{| U| } ^T A_{U {U}^c}(k) \mathbf{1}_{| U^c| }+2 n+\beta
\end{align*}
whenever $\tau_s\leq k_0\leq k_1\leq  \tau_{s+1}$. Since  $\allone_{| U|}^T B_{U U^c}(k)\allone_{| U^c|} =0$ for all $U\subset T$ and $k\geq \tau_{s+1}$, it follows that
$$
    p_0\sum_{k=k_0}^{k_1-1}\allone_{| U^c|}^T B_{U^c U}(k)\allone_{|U|} \leq \sum_{k=k_0}^{k_1-1}\allone_{|U|}^T B_{U U^c}(k)\allone_{|U^c|}+2n+\beta
$$
for all $\tau_s\leq k_0\leq k_1<\infty$. 
This shows that $\{B(k)\}_{k=\tau_s}^{\infty}$ is approximately reciprocal (though one of the associated constants is $\beta+2n$ instead of $\beta$). We can similarly show that $\{C(k)\}_{k=0}^\infty$ is also approximately reciprocal. It can be easily seen that these two sequences also possess the feedback property. Hence, by our inductive hypothesis, there exist positive constants 
\begin{align*}
    \eta_B &:=\min_{r\in [n-1]}\eta_{r}(\gamma,p_0,\beta+2n,\Delta+2n)\cr \text{and } \eta_C &:= \min_{r\in[n-1]}\eta_{r}\left (\gamma,p_0,\beta+\frac{2n+\beta}{p_0},\Delta+\frac{2n+\beta}{p_0}\right)
\end{align*}
such that $B(k_1:k_0)\geq \eta_B I_{ T}$ and $C(k_1:k_0)\geq \eta_C I_{|\bar T|}$ for all $k_0,k_1\in\N_0$ satisfying $\tau_s\leq k_0\leq k_1\leq \tau_{s+1}$. By noting that $A_T(k_1:k_0)\geq B(k_1:k_0)$ and $A_{\bar T}(k_1:k_0)\geq C(k_1:k_0)$, we observe that $A(k_1:k_0)\geq \eta_{\min} I_n$ for all $\tau_s\leq k_0\leq k_1\leq \tau_{s+1}$, where $\eta_{\min}:=\min\{\eta_B,\eta_C\}$. Note that this is true for all $s\in\{0,\ldots,\min\{m, n-1\}\}$ and that the value of $\eta_{\min}$ is independent of $s$.

We now consider two cases.

\textbf{Case 1: $m<n$.} In this case, $\tau_{m+1}$ is defined and it equals $\infty$.  Hence, there exists an $s\in\{0,1,\ldots, m\}$ such that $\tau_s\leq k_1\leq \tau_{s+1}$. Therefore,
\begin{align*}
    A(k_1:k_0) &= A(k_1:\tau_s)\cdot A(\tau_s:\tau_{s-1})\cdots A(\tau_1:\tau_0)\cr
    &\geq \eta_{\min}^{s+1} I_n \geq \eta_{\min}^n I_n.
\end{align*}

\textbf{Case 2: $m=n$.} In this case, $\tau_n<\infty$, so we either have $k_1\leq \tau_n$ or $k_1> \tau_n$.

If $k_1\leq \tau_n$, then there exists an $s\in\{0,1,\ldots,n-1\}$ such that $\tau_s\leq k_1\leq \tau_{s+1}$. Hence, we can proceed as in Case 1. Otherwise, if $k_1>\tau_n$, we need the following analysis.

For each $s\in\{0,1,\ldots,n-1\}$, let $\mathcal G^{(s)}$ be the directed graph whose adjacency matrix $W^{(s)}$ {has entries} given by
$$
    {w}_{ij}^{(s)}=
    \begin{cases}
        1,& \text{if }i {\, = \,} j\text{ {or} } \sum_{k=\tau_{s}}^{\tau_{s+1}-1} {A_{ij}(k)}\geq \frac{1}{n^2},\\
     0,              & \text{otherwise}
    \end{cases},
$$
for all $i,j\in [n]$. We now claim that $\mathcal G^{(s)}$ is a strongly connected digraph for each $s\in\{0,\ldots,n-1\}$. To prove this claim, suppose it is false for some $s\in\{0,\ldots,n-1\}$. Then, there exists a partition $\{T,\bar T\}$ of $[n]$ such that there is no directed link from any node in $T$ to any node in $\bar T$ in $\mathcal G^{(s)}$. This implies $\sum_{k=\tau_{s}}^{\tau_{s+1}-1} \allone_{|\bar T|} ^T A_{\bar T T}(k) \allone_{| T|} =  \sum_{i\in\bar T, j\in T}  \sum_{k=\tau_{s}}^{\tau_{s+1}-1} A_{ij}(k)< |\bar T|\cdot |T|\cdot \frac{1}{n^2}\leq 1$, 
which contradicts the definitions of the times $\tau_0,\ldots, \tau_{n-1}$, thereby proving the claim. Since the weighted adjacency matrix of a strongly connected digraph is irreducible, it follows that $W^{(s)}$ is an irreducible matrix for each $s\in\{0,\ldots,n-1\}$. {As $w_{ii}^{(s)}=1$ for all $i\in [n]$, $W^{(s)}$ is also a primitive matrix~\cite[Page 678]{meyer2000matrix} of the form $W^{(s)}=I_n+Y^{(s)}$, where $Y^{(s)}$ is non-negative for each $s\in\{0,\ldots,n-1\}$. It follows from~\cite{4570093} that $W^{(1)}\cdots W^{(n-1)}$ is positive. Hence, $W^{(0)}W^{(1)}\cdots W^{(n-1)} = (I_n + Y^{(0)})W^{(1)}\cdots W^{(n-1)}$ is a positive matrix.} 

{As a result,} for any two {distinct} indices $i,j\in [n]$, {we have $(W^{(0)}\cdots W^{(n-1)})_{ij}>0$. This implies that}   there exist $r\in [n]$, node indices $l_0\ne l_1\ne \cdots \ne l_r\in [n]$ with $l_0=i$ and $l_r=j$,  and time indices $0\leq s_1\leq s_2\leq \cdots \leq s_{r}\leq n-1$ such that ${w_{l_0 l_0}^{(0)}, \ldots, w_{l_0 l_0}^{(s_1-1)},} w_{l_0 l_1}^{(s_1)}, {w_{l_1 l_1}^{(s_1+1)}, \ldots, w_{l_1 l_1}^{(s_2-1)}}${,} $ w_{l_1 l_2}^{(s_2)}, \ldots,w_{l_{r-1} l_r}^{(s_{r})}{,w_{l_{r} l_r}^{(s_{r}+1)},\ldots, w_{l_r l_r}^{(n-1)}}$ are all positive. From the definition of $W^{(s)}$, it now follows that
\begin{align}\label{eq:many_lower_bound}
    &\sum_{k=\tau_{s_1}}^{\tau_{s_1+1}-1} {A_{il_1}(k)}\geq \frac{1}{n^2},\,\,  \sum_{k=\tau_{s_2}}^{\tau_{s_2+1}-1} {A_{l_1 l_2}(k)}\geq \frac{1}{n^2},\,\, \ldots,\,\, \cr
    &\ldots,\,\,\sum_{k=\tau_{s_{r}}}^{\tau_{s_{r}+1}-1} {A_{l_{r-1} j}(k)}\geq \frac{1}{n^2}.
\end{align}

Next, we bound $A_{l_{u-1} l_{u}}(\tau_{s_u+1}:\tau_{s_u})$ for all $u\in [r]$. On setting $\eta_i = \eta_j = \eta_{\min}$ and  $\delta=\min\{\frac{1}{n^2},\frac{\eta_{\min}}{2}\}$, and then applying Lemma \ref{lem:uncertain} to the sequence $\{A(k)\}_{k=\tau_{s_u}}^{\tau_{s_u+1}}$, we obtain $
    A_{l_{u-1} l_{u}}(\tau_{s_u+1}:\tau_{s_u})\geq \frac{1}{2}\eta_{\min}^2\delta.
$ for each $u\in [r]$.

Now, for any $y\in [n]$ and indices $0\leq s<t\leq m = n$, 
\begin{align}\label{eq:saved}
    A_{yy}(\tau_t:\tau_s)\geq \prod_{k=s}^{t-1}A_{yy}(\tau_{k+1}:\tau_k)\geq \prod_{k=s}^{t-1} \eta_{\min} \geq \eta_{\min}^n.
\end{align}
Thus, {we have the following for all distinct $i,j\in[n]$.}
\begin{align*}
    &A_{ij}(\tau_n:k_0)= A_{ij}(\tau_n:\tau_0)\cr
    &\geq A_{ii}(\tau_{s_1}:\tau_0)\cr
    &\quad\cdot A_{i,l_1}(\tau_{s_1+1}:\tau_{s_1}) A_{l_1 l_1}(\tau_{s_2}:\tau_{s_1+1}) A_{l_1 l_2}(\tau_{s_2+1}:\tau_{s_2})\cdots \cr
    &\quad\cdots A_{l_{r-1} l_{r-1}}(\tau_{s_{r}}:\tau_{s_r}) A_{l_{r-1} j} (\tau_{s_{r}+1}, \tau_{s_{r}}) A_{jj}(\tau_n: \tau_{s_{r}+1 })\cr
    &\geq \left(\eta_{\min}^n\cdot\frac{\eta_{\min}^2\delta}{2} \right)^{r} \eta_{\min}^n \geq \eta_D > 0,
\end{align*}
where $\eta_D:=\left(\eta_{\min}^n\cdot\frac{\eta_{\min}^2\delta}{2} \right)^{n} \eta_{\min}^n$. 
{On the other hand, if $i=j\in [n]$, then using~\eqref{eq:saved} yields $A_{ij}(\tau_n:k_0) =A_{ii}(\tau_n:\tau_0) \ge \eta_{\min}^n \ge \eta_D$.} We have {thus} shown that $A_{ij}(\tau_n:k_0)\geq \eta_D$ for all $i,j\in [n]$, i.e., $A(\tau_n:k_0)\geq \eta_D \allone \allone^T$. Now, \eqref{eq:matrix_prod_ineq} implies 
\begin{align}
    &A(k_1:k_0)=A(k_1:\tau_n) A(\tau_n:k_0)\geq \eta_D A(k_1:\tau_n) \allone_n \allone_n^T\nonumber\\
    &\geq \eta_D e^{-M(\gamma) \Delta} \allone_n \allone_n^T\geq \eta_D e^{-M(\gamma)\Delta}I_n.
\end{align}
To summarize, in both Case 1 and Case 2, we have ${A(k_1:k_0)\geq \eta_F I_n}$ where
$
    \eta_F:= \left(\eta_{\min}^n\cdot\frac{\eta_{\min}^2}{2}\cdot \min\left\{\frac{1}{n^2}, \frac{\eta_{\min}}{2} \right\} \right)^n\eta_{\min}^n e^{-M(\gamma)\Delta}>0.
$
Since $\eta_F$ is uniquely determined by $\gamma,p_0,\beta$ and $\Delta$, it follows that we can define the function $\eta_n:(0,1)\times (0,1)\times (0,\infty)\times[0,\infty)\rightarrow (0,1)$ by the relation $\eta_n(\gamma,p_0,\beta,\Delta)=\eta_F$ while ensuring that 
$
    A(k_1:k_0)\geq \eta_n(\gamma,p_0,\beta,\Delta) I_n
$
 for all $0\leq k_0\leq k_1<\infty$ whenever $\{A(k)\}_{k=0}^\infty$ satisfies the required properties. {Finally, $\eta_n$ is a continuous function, because $\eta_{\min}$, which is determined by $\{\eta_r:r\in[n-1]\}$, is continuous by virtue of our inductive hypothesis.} Thus, the assertion of the proposition holds for $n=q+1$ and hence, for all $n\in\N$.
\end{proof}

\section{Proof of Proposition~\ref{prop:deffuant_weisbuch}}\label{app:C}

\begin{proof}
{To apply our main results, we first need to construct a stochastic chain that captures the given dynamics and then show that the constructed chain is both strongly aperiodic and approximately reciprocal. To this end, 
note that~\eqref{eq:deffuant} can be written as ${x_S(k+1) = A_{2\times2}(k)x_S(k)}$, where we define $S=S(k):=\{\ell(k), m(k) \}$ and}
{\begin{align}\label{eq:auxiliary_deffuant}
    A_{2\times2}(k):= 
    \begin{cases}
        I_2\qquad \quad \quad \text{if}\quad  |x_{\ell}(k) - x_{m(k)}| > r_{\ell(k)\, m(k)},\\
        \begin{pmatrix}
            \alpha(k) & 1- \alpha(k)\\ 
            1-\beta(k) &  \beta(k)
        \end{pmatrix}  \quad \quad\quad\,\,\,\,\,\,\text{otherwise}.
    \end{cases}
\end{align}}
{otherwise. Since no agent other than $i$ and $j$ updates her opinion at time $k$, we also have $x_{\bar S}(k+1) = I_{n-2} x_{\bar S}(k)$, where ${\bar S = \bar S(k):=[n]\setminus S(k)}$. In light of this, the above implies that $x(k+1) = A (k) x(k)$
where $A(k)\in \R^{n\times n}$ is  defined by the following conditions on its sub-matrices: ${A_{S}(k)= A_{2\times2}(k)}$,  $A_{\bar S}(k) = I_{n-2}$, $A_{S\bar S}=O_{2\times (n-2)}$ and $A_{\bar S S} = O_{(n-2)\times 2}$.}

Now,~\eqref{eq:auxiliary_deffuant} and $x(k+1) = A (k) x(k)$ together imply   
$
     a_{ii}(k) \ge \min\{1,\alpha(k),\beta(k)\}\ge \delta
$ \textit{a.s.} for all $i\in [n]$ and $k\in\N_0$. Hence, $\{A(k)\}_{k=0}^\infty$ is strongly aperiodic almost surely.

{To establish approximate reciprocity, consider any two times $k_1,k_2\in\N_0$ with $k_1<k_2$, and consider the following cases.}

{Case 1: $k_2-k_1\le B$. In this case, we have  $$\sum_{k=k_1}^{k_2-1} \allone_{|T|} ^T A_{T\bar T}(k) \allone_{|\bar T|} \le \sum_{k=k_1}^{k_2-1} \allone_{n} ^T A(k) \allone_{n} = n(k_2-k_1)\le nB.$$
for all $T\subset[n]$, i.e.,~\eqref{eq:crucial} holds for $\beta = nB$ and all $p_0\in(0,1)$.} 

{Case 2: $k_2 - k_1> B$. In this case, let $K:=\{k_1,\ldots, k_2\}$ and let $I\subset K$ be as defined in the proposition. We then have
\begin{align}\label{eq:nB_two}
    \sum_{k\in K\setminus I} \allone_{|T|} ^T A_{T\bar T}(k) \allone_{|\bar T|} &\le \sum_{k\in K\setminus I} \allone_{n} ^T A(k) \allone_{n}\cr
    &= n |K\setminus I|\stackrel{(a)}\le nB,
\end{align}
for all $T\subset [n]$, where $(a)$ holds because it is given that $|I|\ge k_2-k_1-B= |K|-B$. On the other hand, for $k\in I$, we bound $\sum_{k\in I} \allone_{|T|} ^T A_{T\bar T}(k) \allone_{|\bar T|}$ for each $T\subset [n]$ as follows: we first define  $I_{\text{ess}}(T)$ as the set of all times $k\in I$ at which (a) opinion updates take place, and (b) exactly one of $\ell(k)$ or $m(k)$ belonging to $T$. That is, $I_{\text{ess}}(T):= I_1(T)\cap I_2(T)$, where $I_{1}(T):=\{ k\in I: |x_{\ell}(k) - x_{m(k)}| \le r_{\ell(k)\, m(k)}\}$ and $I_2(T):=\{k\in I: \ell(k)\in \bar T, m(k) \in  T\} \cup \{k\in I: \ell(k)\in  T, m(k) \in  \bar T\}$.
We then observe that for all $k\notin \Iess(T)$, we  either have $T\subset \bar S(k)$, $\bar T\subset  \bar S(k)$, or $A_{2\times 2}(k)=I_2$, in each of which cases it follows from the definition of $A(k)$ that $\allone_{|T|} ^T A_{T\bar T}(k) \allone_{|\bar T|}  = 0$. Consequently,
\begin{align}\label{eq:epsilon}
    \sum_{k\in I} \allone_{|T|} ^T A_{T\bar T}(k) \allone_{|\bar T|} &= \sum_{k\in  \Iess(T)} \allone_{|T|} ^T A_{T\bar T}(k) \allone_{|\bar T|}\cr
    &\stackrel{(a)}\ge \sum_{k\in \Iess(T)} \min\{1-\alpha(k),1-\beta(k)\} \cr
    &\ge \sum_{k\in \Iess(T)} \varepsilon \,\,=  \varepsilon |\Iess(T)|
\end{align}
\textit{a.s.}, where $(a)$ holds because $A(k)$ has exactly two  non-zero off-diagonal entries and the definition of $\Iess(T)$ implies that exactly one of these is in $A_{T\bar T}(k)$. Similarly, we can show that $\sum_{k\in I} \allone_{|T|} ^T A_{T\bar T}(k) \allone_{|\bar T|} \le (1-\delta)|\Iess(T)|$ holds \textit{a.s.}}
{Since $T\subset[n]$ is arbitrary, combining this with~\eqref{eq:epsilon} yields
\begin{align}\label{eq:almost_end}
    \left(\frac{\varepsilon}{1-\delta}\right) \sum_{k\in I}\allone_{|T|} ^T A_{T\bar T}(k) \allone_{|\bar T|} \le \sum_{k\in I} \allone_{|\bar T|} ^T A_{\bar T T}(k) \allone_{|T|}.
\end{align}
Multiplying both sides of~\eqref{eq:nB_two} by $\frac{\varepsilon}{1-\delta}$ and combining the result with~\eqref{eq:almost_end} culminates in
$\varepsilon(1-\delta)^{-1} \sum_{k\in K}\allone_{|T|} ^T A_{T\bar T}(k) \allone_{|\bar T|} \le \sum_{k\in K} \allone_{|\bar T|} ^T A_{\bar T T}(k) \allone_{|T|} +  \varepsilon(1-\delta)^{-1}nB$ \textit{a.s.}, which shows that~\eqref{eq:crucial} holds with $p_0=\frac{\varepsilon}{1-\delta}$ and $\beta=\frac{\varepsilon nB}{1-\delta}$.} 

{We have thus shown that almost every realization of $\{A(k)\}_{k=0}^\infty$ is both strongly aperiodic and approximately reciprocal.} It now follows from Remark~\ref{remark:new_remark}  that almost every realization of ${\lim_{k\to\infty} A(k:0)}$ exists. Consequently, ${\lim_{k\to\infty}x(k)=\lim_{k\to\infty} A(k:0)x(0)}$ exists \textit{a.s.} for all initial conditions $(k_0,x(k_0))\in\N_0\times \R^n$.
\end{proof}

\section{Proof of Inequality~\eqref{eq:extra}}\label{app:nine}
\begin{proof}
     {Let $\tau_{ij}:=|\{k\in\{k_0,k_0+1,\ldots, k_1-1\}:a_{ij}(k)>0\}|$ denote the number of times sensor $j\in [n]$ transmits to sensor $i\in [n]\setminus\{j\}$ during the given time period.  We can express $\tau_{ji}$ as $\tau_{ji}=qT+r$ for some $q\in\N_0$ and $r\in \{0,1,\ldots, T-1\}$. The definition of $T$ then implies that $q = \lfloor \tau_{ji}/T \rfloor \le \tau_{ij}$. Equivalently, $\tau_{ji}\le T\tau_{ij}+r$. Since $r\le T-1$, this further implies $\tau_{ji}\le T\tau_{ij}+(T-1)$, which is in turn equivalent to $\tau_{ij}\ge T^{-1}\tau_{ji} - (T-1)T^{-1}$. This observation enables us to  derive the following chain of inequalities:}
    {
\begin{align} \label{eq:soon}&\sum_{k=k_0}^{k_1-1} \allone_{|\bar S|} ^T A_{\bar S S}(k)\allone_{|S|} =\sum_{i\in \bar S, j\in S} \sum_{k=k_0}^{k_1-1} a_{ij}(k) 1_{\{ a_{ij}(k)>0\} } \cr &\ge \sum_{i\in \bar S, j\in S} \sum_{k=k_0}^{k_1-1} \delta 1_{\{ a_{ij}(k)>0\} } = \sum_{i\in \bar S, j\in S} \delta  \tau_{ij}\cr &\stackrel{(a)}\ge \delta T^{-1} \sum_{i\in \bar S, j\in S} \tau_{ji} - (T-1)T^{-1}\delta\sum_{i\in \bar S,j\in S} 1\cr &=\delta T^{-1}\sum_{i\in \bar S, j\in S} \sum_{k=k_0}^{k_1-1} 1_{\{ a_{ji}(k)>0\} } - (T-1)T^{-1}\delta\sum_{i\in \bar S,j\in S} 1\cr &\stackrel{(b)}\ge \delta T^{-1}\sum_{i\in \bar S, j\in S} \sum_{k=k_0}^{k_1-1} a_{ji}(k) 1_{\{ a_{ji}(k)>0\} } - (T-1)T^{-1}\delta n^2\nonumber\\ &=\delta T^{-1}\sum_{k=k_0}^{k_1-1} \allone^T_{|S|} A_{S\bar S}(k)\allone_{|\bar S|} - (T-1)T^{-1}\delta n^2, \end{align} where   $1_{\{ a_{ij}(k)>0\} }\in\{0,1\}$ denotes an indicator variable that equals 1 if and only if $a_{ij}(k)>0$, $(a)$ follows from the observation ${\tau_{ij}\ge T^{-1}\tau_{ji} - (T-1)T^{-1}}$, and $(b)$ follows from the facts that $1\ge a_{ji}(k)$ and $\sum_{i\in \bar S, j\in S} 1 = |S||\bar S|\le n^2$. Finally, we note that~\eqref{eq:soon} is equivalent to~\eqref{eq:extra}, as required.}
\end{proof}

\end{document}